\documentclass{article}

\usepackage[margin=1in]{geometry}

\usepackage{graphicx}
\usepackage{lipsum,cite}
\usepackage{amsfonts,amssymb,amsmath,amsthm}
\usepackage{bm}
\usepackage{mathtools}
\usepackage{nicematrix}
\usepackage{epstopdf}
\usepackage{algorithmic}
\usepackage[algo2e,ruled,vlined,linesnumbered]{algorithm2e}
\usepackage{xcolor}
\usepackage{marginnote}
\usepackage{authblk}
\usepackage{hyperref}
\usepackage[nameinlink]{cleveref}
\usepackage{doi}
\usepackage{xr}

\DeclareGraphicsExtensions{.eps,.pdf,.png,.jpg}

\title{Birkhoff interpolation models for optimization\\with some available derivatives\thanks{This work was supported by the U.S.~Department of Energy, Office of Science, Office of Advanced Scientific Computing Research and Office of Fusion Energy Science, Scientific Discovery through Advanced Computing (SciDAC) Program through the FASTMath Institute and the StellFoundry Partnership Project under Contract No.~DE-AC02-06CH11357.}}

\author[1]{Jeffrey Larson}
\author[1]{Matt Menickelly}
\author[1]{Evan Toler}
\affil[1]{Argonne National Laboratory, Lemont, IL 60439\\
\texttt{\{jmlarson,mmenickelly,etoler\}@anl.gov}}

\date{}

\newtheorem{theorem}{Theorem}
\newtheorem{lemma}[theorem]{Lemma}

\newtheorem{proposition}[theorem]{Proposition}
\newtheorem{corollary}[theorem]{Corollary}

\theoremstyle{definition}
\newtheorem{assumption}{Assumption}
\newtheorem{definition}[theorem]{Definition}

\theoremstyle{remark}

\crefname{hypothesis}{Hypothesis}{Hypotheses}
\crefname{assumption}{assumption}{assumptions}
\Crefname{assumption}{Assumption}{Assumptions}
\crefname{algocf}{algorithm}{algorithms}
\Crefname{algocf}{Algorithm}{Algorithms}
\crefname{algline}{line}{lines}
\Crefname{algline}{Line}{Lines}
\crefformat{algline}{#2Line~#1#3}
\Crefformat{algline}{#2Line~#1#3}

\newcommand{\lineofalg}[2]{\hyperref[#2]{Line \ref{#1} of \Cref{#2}}}

\makeatletter
\renewcommand{\eqref}[1]{%
  \begingroup
    \leavevmode
    \hyperref[#1]{(\ref*{#1})}%
  \endgroup
}
\makeatother

\usepackage{amsopn}

\DeclareMathOperator*{\argmax}{arg\,max}

\newcommand{\minimize}{\operatornamewithlimits{minimize}}

\newcommand{\cL}{\mathcal{L}}
\newcommand{\cP}{\mathcal{P}}
\newcommand{\cX}{\mathcal{X}}

\newcommand{\ab}{\mathbf{a}}
\newcommand{\bb}{\mathbf{b}}
\newcommand{\cb}{\mathbf{c}}
\newcommand{\eb}{\mathbf{e}}
\newcommand{\gb}{\mathbf{g}}

\newcommand{\sv}{\mathbf{s}}
\newcommand{\tb}{\mathbf{t}}
\newcommand{\vb}{\mathbf{v}}
\newcommand{\wb}{\mathbf{w}}
\newcommand{\xb}{\mathbf{x}}
\newcommand{\yb}{\mathbf{y}}
\newcommand{\zb}{\mathbf{z}}
\newcommand{\zerob}{\mathbf{0}}

\newcommand{\Bb}{\mathbf{B}}
\newcommand{\Cb}{\mathbf{C}}
\newcommand{\Eb}{\mathbf{E}}
\newcommand{\Ib}{\mathbf{I}}
\newcommand{\Hb}{\mathbf{H}}
\newcommand{\Mb}{\mathbf{M}}
\newcommand{\Qb}{\mathbf{Q}}
\newcommand{\Wb}{\mathbf{W}}

\newcommand{\alphab}{\pmb{\alpha}}
\newcommand{\betab}{\pmb{\beta}}
\newcommand{\phib}{\pmb{\phi}}

\newcommand{\lambdab}{\pmb{\lambda}}
\newcommand{\Deltamax}{\Delta_{\mathrm{max}}}
\newcommand{\gammainc}{\gamma_{\mathrm{inc}}}
\newcommand{\gammadec}{\gamma_{\mathrm{dec}}}
\newcommand{\kappaeh}{\kappa_{\mathrm{eh}}}
\newcommand{\kappaeg}{\kappa_{\mathrm{eg}}}
\newcommand{\kappaef}{\kappa_{\mathrm{ef}}}

\renewcommand{\vec}[1]{\mathbf{#1}}
\newcommand{\defined}{\triangleq}
\newcommand{\Ball}{\mathcal{B}}
\newcommand{\norm}[1]{\left\| #1 \right\|}
\newcommand{\normil}[1]{\| #1 \|} %
\newcommand{\abs}[1]{\left| #1 \right|}
\newcommand{\Reals}{\mathcal{R}}

\newcommand{\grad}{\nabla}
\newcommand{\Hess}{\nabla^2}
\newcommand{\textvec}[1]{\textbf{vec}(#1)}

\hypersetup{
  breaklinks=true,
  colorlinks=true,
  linkcolor={blue!90!black},
  citecolor={blue!90!black},
  urlcolor={blue!90!black},
  pdftitle={Birkhoff interpolation models for optimization with some available derivatives},
  pdfauthor={J. Larson, M. Menickelly, E. Toler}
}

\begin{document}

\maketitle

\begin{abstract}
We consider interpolation-based derivative-free optimization in settings where only some derivatives are available.
Such situations arise in scientific computing applications involving simulations, adjoint-enabled components, legacy software, or partially differentiable models.
We introduce a Birkhoff interpolation framework that permits arbitrary patterns of derivative availability and enables the construction of local polynomial models using mixtures of function values and partial derivative information.
In contrast to Hermite interpolation approaches, the proposed framework does not require all available derivatives to be queried at every interpolation point.
We develop conditions under which the resulting interpolation systems are poised and establish corresponding model-accuracy bounds for fully quadratic interpolation models.
We develop a trust-region framework that maintains poised interpolation sets while selectively incorporating derivative information.
The method generalizes an established class of interpolation-based derivative-free optimization algorithms and naturally bridges derivative-free and derivative-based settings.
We evaluate our approach on a collection of CUTEst test problems with synthetically generated derivative-availability patterns, as well as a stellarator design problem in which only some derivatives are available.
\end{abstract}

\noindent\textbf{Keywords.}
Birkhoff interpolation, derivative-free optimization, trust-region methods

\noindent\textbf{MSC codes.}
68Q25, 90C56, 41A05

\section{Introduction}
\label{sec:introduction}
This manuscript considers numerical optimization problems where partial
derivatives with respect to all variables of the objective function are not
readily available.
In particular,
we consider unconstrained problems of the form
  \begin{equation}\label{eq:problem_def}
    \begin{aligned}
      \minimize_{\xb \in \Reals^n}\; & f(\xb), \\
    \end{aligned}
  \end{equation}
where we have access  to an oracle that returns the objective function value $f(\xb)$, as well as (via a separate oracle call), $\partial_{x_i} f$, that is,
the partial derivatives
 with respect to some (but not all) entries $x_i$ of $\xb$.
We assume, for each $i$, that the cost of evaluating $\partial_{x_i} f$ is comparable to the cost of evaluating $f$.
For instance, in a simulation-based setting, each scalar sensitivity $\partial_{x_i} f$ might be obtained via a separate algorithmic differentiation or complex-step perturbation, each costing roughly as much as a single evaluation of $f$. This is distinct from settings in which a single adjoint call returns several (or all) partial derivatives at the cost of essentially a small constant number of evaluations of $f$.
In this manuscript, we restrict attention to the unconstrained case illustrated in \cref{eq:problem_def}; \Cref{sec:conclusions} discusses constrained extensions.

Problems of this form exist in numerous scientific domains, including
stellarator design in the magnetic confinement fusion community, and yield
optimization in product manufacturing~\cite{Jorge2023, Landreman_2021, Dudt_Conlin_Panici_Kolemen_2023, Fuhrlander2023, Graeb2007}.
When designing a nuclear fusion stellarator, it is typical to model one set of decision variables
corresponding to the shape of the interior plasma fuel (the ``plasma parameters'') and a second set of decision variables
corresponding to the shapes of the exterior confining magnetic coils (the ``coil parameters").
Common objective functions $f$ throughout the stellarator
optimization literature assess the viability of a design
for given plasma and coil parameters~\cite{Jorge2023}.
Frequently, derivatives of $f$ with
respect to the coil parameters are relatively inexpensive to
compute and are hence available via an oracle.
On the other hand, derivatives of $f$ with respect to the plasma
parameters are either much more expensive to evaluate or no oracle for their computation has been developed.

Recent advances in algorithms have led to the development of methods
that partition the decision
variables to exploit structural properties of the objective function~\cite{bouchet2024,
lapucci2024}.
This work is similar in motivation, where the partition criterion is derivative availability.
The practical motivation for this is many scientific optimization problems where derivative information is limited across variables or model components. Certain quantities may admit inexpensive exact derivatives through adjoint or automatic differentiation techniques, while others are not readily available because they are computed through legacy software.
We seek a framework where available derivative information can be incorporated selectively, that is, without requiring access to a full gradient or Hessian.
This flexibility enables the construction of higher-quality local models while avoiding the computational expense (or unavailability) of derivatives.

Liuzzi and Risi previously described a method for problems that partition variables based on derivative availability.
Their method generated a sequence of alternating directions, generating a step based on finite differences in the block of unknown derivatives followed by a step based on gradients in the block of known derivatives~\cite{Liuzzi_2010}.
The present work differs from this approach by optimizing jointly in the known and unknown directions.
Moreover, Liuzzi and Risi limit their discussion to first-order derivatives, whereas we discuss settings that can include higher orders of derivative information when it is available.
Cecere et al. developed a penalty decomposition method for finite-sum, partially separable optimization problems in which first-order information is only partially accessible~\cite{Cecere_2025}.
The present work, by contrast, considers the scenario where not all gradient components are available and the objective need not be expressible as a finite sum; however, if the objective function is a finite sum, our method assumes that the known partials are indeed known for each summand.
Thomann and Eichfelder consider a related setting in multiobjective optimization, where all objectives but one have available gradient oracles~\cite{Thomann2019}.

An approach more closely related to ours is given by Fuhrl\"{a}nder and Sch\"{o}ps, who
proposed modifications to the derivative-free trust-region method \texttt{BOBYQA} to
incorporate partial derivative information~\cite{Fuhrlander2023}. Their method
augments the interpolation-based quadratic models employed in
\texttt{BOBYQA} by including available gradient information in the interpolation conditions of the (now, generally speaking, regression) models.
This approach enforces Hermite-type conditions in model construction.

Related ideas appear also in earlier work on hybrid and partially separable
optimization. For instance, the framework of Colson and Toint for partially
separable derivative-free problems~\cite{Colson_2005} could be extended to the case where
the objective function restricted to a separable block admits an oracle that can compute all corresponding partial derivatives.
Abramson et
al.~examined incorporating derivative sign information to prune search
directions in generalized pattern search algorithms~\cite{Abramson_2003}.
Although not explicitly
stated, if oracle access to only a subset of partial derivatives were available, the technique of Abramson et al. could easily be adapted to this setting.
Bertolli et al.~adapted a sequential
quadratic programming framework to use available derivatives with respect to
some (but not all) parameters within a bi-level optimization problem for a nuclear
physics application~\cite{Bertolli_2012}.

Our present work can be seen as a generalization of the derivative-free optimization approach described by Conn, Scheinberg, and Vicente~\cite{CSV2009,Conn2009}.
In (one version of) that approach, one locally approximates the objective $f$ with a model $m$ using past evaluations of $f$ at a set of points satisfying a set of Lagrange interpolation conditions.
That is, the objective and the local model must take the same value at a prescribed, finite collection of interpolation nodes.
With this approximation, there are extensive theoretical results describing, for example, desirable geometries for the interpolation nodes; the supremum of the error $|f - m|$ over a given set (in particular, a trust region); and, when the approximation is employed within a trust-region method, convergence of the method to a stationary point.
Within this framework, algorithms also exist for generating poised interpolation conditions and model improvement.
Our work extends these theoretical results and algorithms to the setting where some derivatives are available and can be incorporated into the interpolation conditions for the local approximation model $m$.
We develop analysis that builds upon the classical work of Ciarlet and Raviart, which describes error bounds for polynomial interpolation based on Taylor series~\cite{Ciarlet_1972}.
In particular, they discuss the Birkhoff interpolation setting, which is a variation on the classical Hermite interpolation problem.

Both Birkhoff and Hermite interpolation of a function $f$ refers to settings where one jointly uses evaluations of $f$ and its derivatives to construct an interpolant.
However, Hermite interpolation classically assumes that if any derivative is interpolated at a point, then all derivatives of lower order must also be interpolated at the same point.
Birkhoff interpolation, by contrast, makes no such restriction. One may require, for example, that a model interpolates a first derivative value at a point without also interpolating the function value there.
We will focus on the Birkhoff interpolation setting and  use  the extended flexibility it allows compared with Hermite interpolation.

The paper is organized as follows.
We give background to the some-available-derivatives setting \Cref{sec:background}.
We then present our algorithm in \Cref{sec:alg}, analyze its convergence properties in \Cref{sec:converge}, and show experimental results in \Cref{sec:experiments}.
\Cref{sec:conclusions} concludes with a discussion of next steps.

\section{Notation}
The extension from Lagrange interpolation to Birkhoff interpolation requires an extension of notation to discuss the variety and nature of available derivatives.
Therefore, we introduce the notation here in order to ease the presentation of future material for the reader.

Multi-index notation is particularly useful in the present setting because interpolation conditions may involve arbitrary collections of mixed partial derivatives.
The resulting Birkhoff interpolation framework naturally accommodates limited derivative availability patterns across variables and derivative orders.

Let $\Reals$ and $\mathcal{Z}_+$ denote the set of real numbers and positive integers, respectively.
In this manuscript, bold characters (e.g., $\vec{x}$) denote vector and matrix
quantities, and subscripts on non-bold versions of the same symbol (e.g., $x_i$) denote the corresponding scalar entries.
$\Ib_r$ denotes the identity matrix of size $r \times r$.
Superscripts on vectors denote indices of points within a set or multiset (e.g., $\{\yb^1, \yb^2, \dots, \yb^p\}$).
Capital Roman letters like $A$ denote finite sets, and capital calligraphic letters like $\Reals$ denote possibly infinite sets.
Algorithmic parameters and multi-indices are represented by lowercase Greek letters.
Superscripts in parentheses (e.g., $\xb^{(0)}$) indicate an iteration counter for a given term within the context of an iterative method.
Throughout, $n$ denotes the dimension of the domain of $f$.

Given a multi-index $\alphab = [\alpha_1, \dots, \alpha_n]$ of
nonnegative integer entries, we associate a differential operator that acts on the $x_j$ coordinate $\alpha_j$ times,
\[
\partial^{\alphab}
\defined
\partial_{x_1}^{\alpha_1} \dots \partial_{x_n}^{\alpha_n}
.
\]
We denote the total order of differentiation by
$\abs{\alphab} \defined \sum_{j=1}^n \alpha_j.$
For example, in $\Reals^3$, the first-order partial derivative with respect to $x_1$ is denoted by $\alphab = [1, 0, 0]$,
the second-order partial derivative with respect to $x_1$ twice is denoted by $\alphab = [2, 0, 0]$,
and the second-order partial derivative with respect to $x_1$ once and $x_2$ once is denoted by $\alphab = [1, 1, 0]$.
We let $A$, which we call the \emph{available set}, denote the finite set of multi-indices corresponding to the
derivative information available for interpolation. Thus, $\alphab \in A$
indicates that values of $\partial^{\alphab} f$ may be used in constructing the
model; in particular, the multi-index $\zerob \in A$ whenever function values
are available.

The approach we develop in this manuscript employs local polynomial models of $f$. %
While this manuscript focuses on, and analyzes, the specific case where these polynomial models are quadratics, analogous statements and results can be made for other classes of models, for example, radial basis functions.
Throughout, we consider the natural (i.e., monomial) basis
\[
\phi \defined \left\{ 1, x_1, \dots, x_n, \frac{1}{2} (x_1)^2, x_1 x_2, x_1 x_3, \dots, x_{n-1} x_{n}, \frac{1}{2} (x_n)^2 \right\}
\]
for the space of polynomials of degree at most 2 in $n$ dimensions, which we denote $\cP_n^2$.
We frequently refer to the number $q \defined (n + 1)(n + 2)/2 - 1$ and enumerate the natural basis functions as $\phi_0, \dots, \phi_q$.
We denote the evaluation of a degree-2 monomial in $\phi$ at a given point $\yb\in\Reals^n$ by
\[
\phib_{\text{quad}}(\yb) \defined
\begin{bmatrix}
    \frac{1}{2} (y_1)^2 & y_1 y_2 & y_1 y_3 & \dots & y_{n-1} y_{n} & \frac{1}{2} (y_n)^2
\end{bmatrix}
^T.
\]
Similarly, we write the evaluation of all natural basis functions at $\yb$ as
\begin{equation} \label{eq:basis_fun_def}
\phib(\yb) \defined
\begin{bmatrix}
    1 \\
    \yb \\
    \phib_{\text{quad}}(\yb)
\end{bmatrix}
\in \Reals^{q+1}.
\end{equation}

We let $\textvec{\Hb}$ denote the upper triangular components of a symmetric matrix $\Hb$
enumerated as a vector.
When $\Hb$ is the Hessian matrix of a quadratic polynomial, the order of the vectorization follows the order of the natural basis functions $\phib_{\text{quad}}$.

\section{Background} \label{sec:background}

The optimization algorithm we present here employs multivariate
interpolation models.
We follow the setup for interpolation models employed within the context of model-based derivative-free optimization methods as described by Conn, Scheinberg, and Vicente~\cite{CSV2009}.
However, this existing setup assumes access only to a zeroth-order oracle;  that is, it assumes only the existence of an oracle corresponding to an available set $A=\{\zerob\}\in\mathcal{Z}_{+}^n$.
We extend their treatment of Lagrange interpolation models, appropriate for the $A=\{\zerob\}$ case, to Birkhoff
interpolation models, which are appropriate for more general available sets $A$.

\subsection{Birkhoff Interpolation} \label{sec:birkhoff_interp}

We construct models of the objective function $f$ by Birkhoff interpolation.
In Birkhoff interpolation, one prescribes points $\yb$ where the interpolation model $m$ should match the value of $f$ or one of its derivatives.
We now define the Birkhoff interpolation problem for a quadratic polynomial model.
In this problem, \emph{interpolation data} is provided via a set of pairs
\[
D = \{ (\yb^i, \alphab^i) \mid i = 0, \dots, q \}
.\]
In this manuscript, the Birkhoff interpolation problem consists of finding a quadratic polynomial $m$ centered at the point $\yb^0$,
\begin{align}
    m(\yb^0 + \sv) &= c + \gb^T \sv + \frac{1}{2} \sv^T \Hb \sv, \text{ such that } 
    \label{eq:m_expression} \\
    \partial^{\alphab^i} m(\yb^i) &= \partial^{\alphab^i} f(\yb^i), 
    \qquad \forall i = 0, \dots, q.
    \label{eq:m_interp}
\end{align}
We refer to $m$ satisfying both \eqref{eq:m_expression} and \eqref{eq:m_interp} as a \emph{Birkhoff interpolation model} of $f$ for the interpolation data $D$.
The center $\yb^0$ (the first argument of the first entry of $D$), always throughout this manuscript, refers to the center of $m$, and will update concomitantly with $m$.
In \Cref{assn:A_D} below, we restrict the available sets $A$ to satisfy $\abs{\alphab} \le 2$ for all $\alphab \in A$.
We note that if $\alphab^i=\zerob^n$ for every $i$, then \eqref{eq:m_interp} recovers the Lagrange interpolation conditions;
thus, Lagrange interpolation may be viewed as a special case of Birkhoff
interpolation.
Birkhoff interpolation is also closely related to Hermite interpolation:
Hermite interpolation refers to the setting where $(\yb, \alphab) \in D$ implies that $(\yb, \betab) \in D$ for all $\betab$ such that $\beta_j \le \alpha_j$ for all $j$.
In words, Hermite interpolation requires that if any derivative is interpolated at a point $\yb$, then every lower-order derivative is also interpolated at $\yb$.
We do not put such a restriction on the conditions \eqref{eq:m_interp} and
allow, for example, a derivative interpolation condition at a point without requiring a
zeroth-order interpolation condition to hold at the same point.
Thus, Birkhoff interpolation also generalizes Hermite interpolation.

It is permitted in the Birkhoff interpolation problem to interpolate multiple derivative values at the same point, or to interpolate the same derivative multi-index at multiple points (as in the case in Lagrange interpolation).
As such,
$Y = \left\{ \yb^i | (\yb^i, \alphab^i)  \in D \right\}$
is a multiset that may contain duplicate elements. 
\footnote{Though $Y$ is formally a multiset, we will use the term ``set" going forward for brevity.}
We note that the conditions \eqref{eq:m_interp} and the data $D$ refer to $q+1$ inseparable pairs $(\yb^i, \alphab^i)$; it is insufficient to specify only the unordered information $Y$ and $A$.

Throughout the manuscript, we distinguish between the current working set of points and the accumulated history of all sampled data over the course of optimization.
$D$ will denote a finite set of $q+1$ interpolation data (points and $\alphab$) that is used to construct a local model.
In contrast, $\bar{D}$ denotes the full history of data observed over the course
of the algorithm. In general, this set will be larger than $D$ and is not used in its entirety at each iteration.
In some contexts, we also introduce $D'$ to represent a modified or candidate
set derived from $D$ (e.g., through replacement/augmentation/improvement steps).
This distinction allows us to separate the roles of model construction (which will act on $D$ and
$D'$) from total data accumulated (stored in $\bar{D}$).
The  points $Y$ are always assumed to be associated with the set $D$ in whichever context it appears.

One strength of modeling based on available derivative information is that the corresponding optimization algorithm could potentially take advantage of disparate costs between different derivatives; this is, in fact, our key motivation.
If each available $\alphab$ has a corresponding computational cost $T_{\alphab}$ to evaluate $\partial^{\alphab} f$, then one can incorporate the relative values of $T_{\alphab}$ as a factor to determine which interpolation conditions to use during model building.
For example, given an idealized model of memory access, reverse-mode algorithmic differentiation provably exhibits $T_{\alphab}$ as a small constant (bounded above by 5) multiple of $T_{\zerob}$; see, for example,~\cite[Chapter 4]{griewank2008edp}.
We stress that, in the present manuscript, we do not take advantage of this property in our numerical experiments, but we acknowledge this as a  future direction of work.

In many applications, the available set $A$ is determined by the computational structure of the underlying simulation code.
For example, derivatives with respect to geometry parameters may be available through adjoint methods, while derivatives associated with legacy software components may not be accessible. In such cases, the available set $A$ is known prior to optimization.
For future development, we also note that given a data-driven cost model for obtaining derivatives, one may be well-advised to include any computable derivative in the available set $A$ and let the cost model dictate which derivatives are ever actually evaluated during an optimization; however, we again stress that the numerical tests of this paper will not employ such a sophisticated cost model.

We emphasize that our framework does not require derivative evaluations to have negligible cost relative to function evaluations.
Rather, the setting of interest is one in which a subset of derivatives can be obtained at a substantially lower marginal cost than would be required to compute complete derivative information.
Our algorithm is intended for settings where partial derivative information improves local model quality without fundamentally changing the dominant computational expense.

\subsection{Assumptions}
We will cast two assumptions concerning problem \eqref{eq:problem_def}.
The first assumption ensures that $f$ is suitable for defining a particular class of Birkhoff interpolation problems.
The second ensures that $f$ is sufficiently well-behaved to guarantee convergence of our algorithm to a critical point.

Toward the first assumption, we need to assume that the available set $A$ is known and given.
As stated in \eqref{eq:m_expression}, we assume that the Birkhoff interpolation model is a quadratic polynomial.
Thus, we assume that the available set $A$ is limited to $\alphab$ satisfying $\abs{\alphab}\leq 2$.
The results and analysis for linear interpolation models would follow as straightforward analogues.

From \eqref{eq:m_expression} and the interpolation conditions \eqref{eq:m_interp}, we see that at least one interpolation condition must be a Lagrange condition; that is, $\alphab^i = \zerob$ for some $i$.
Otherwise, it is impossible to determine the constant $c$ in \eqref{eq:m_expression}.
We assume, without loss of generality, that the model center $\yb^0$ is such a point satisfying $(\yb^0,\zerob)\in D$.
Because our trust-region algorithm will need to have computed the value of the model center $f(\yb^0)$ on every iteration in order to determine step acceptance, it is practical to assume that the value of $f(\yb^0)$ is available at the outset of every iteration.
Imposing a Lagrange interpolation condition at the model center $\yb^0$ also simplifies to the condition that $c = f(\yb^0)$, where $c$ is the constant term in \eqref{eq:m_expression}.
Although for clarity we will not make this simplification in our analysis, in practice one can remove $c$ as a degree of freedom in determining the model coefficients through an appropriate shifting of interpolation data.

In the following assumption we summarize these necessary conditions on the available derivative set and the interpolation data.

\begin{assumption} \label{assn:A_D}
    Assume that the interpolation model $m$ is a quadratic polynomial of the form \eqref{eq:m_expression}.
    Assume that $\zerob \in A$ and that $\abs{\alphab} \le 2$ for all $\alphab\in A$.
\end{assumption}

To guarantee convergence of model centers
to a point exhibiting necessary optimality conditions, we make (standard) assumptions on the objective function $f$.
We will assume that $f$ is sufficiently smooth and bounded below.
Since we are designing a monotonic trust-region algorithm, iterates are only accepted provided they yield a decrease in an incumbent objective value.
Consequently, all model centers lie in the sublevel set $\cL(\xb^{(0)}) \defined \{ \xb \in \Reals^n \mid f (\xb) \le f(\xb^{(0)})\}$.
We also impose a maximum trust radius $\Deltamax$, which ensures $f$ is only ever evaluated on the set
\[
\cL'(\xb^{(0)}) \defined \bigcup_{\xb \in \cL(\xb^{(0)})} \Ball(\xb, \Deltamax),
\]
where $\Ball(\xb, \Deltamax)$ is the ball of radius $\Deltamax$ centered at $\xb$.
This is all summarized in the following assumption.
\begin{assumption} \label{assn:poised_D2lipshitz}
    Let $\grad f$ and $\Hess f$ be Lipschitz continuous on $\cL'(\xb^{(0)})$
    with respective Lipschitz constants $0 < L_{\grad f}, L_{\Hess f} < \infty$.
    Assume that there exists a lower bound $f_\star$ such that $f(\xb) \ge f_\star$ for all $\xb \in \cL'(\xb^{(0)})$.
\end{assumption}

\subsection{Linear Algebra Concerns for Birkhoff Interpolation}
We now represent the interpolation conditions \eqref{eq:m_interp} as a linear system for the model coefficients.
With this convention, and recalling that the model $m$ is centered at $\yb^0$,
the Birkhoff interpolation conditions for the model coefficients in
\eqref{eq:m_expression} may be written as
\begin{equation} \label{eq:m_interp_linsys}
    \begin{bmatrix}
        \partial^{\alphab^0} \phib(\yb^0 - \yb^0)^T \\
        \vdots \\
        \partial^{\alphab^{q}} \phib(\yb^{q} - \yb^0)^T
    \end{bmatrix}
    \begin{bNiceArray}{c}
        c \\
        \gb \\
        \textvec{\Hb}
    \end{bNiceArray}
    =
    \begin{bmatrix}
        \partial^{\alphab^0} f(\yb^0) \\
        \vdots \\
        \partial^{\alphab^{q}} f(\yb^{q})
    \end{bmatrix}
    ,
\end{equation}
where $\phib$ is the vector of basis function evaluations \eqref{eq:basis_fun_def}.
We denote the matrix in the system \eqref{eq:m_interp_linsys} as $\Mb = \Mb(\phi, D)$.
Define
\[
\Delta(Y) \defined
\begin{dcases}
    \max_i {\left\| \yb^i - \yb^0 \right\|} \quad &\exists \yb \in Y \text{ with } \yb \ne \yb^0, \\
    \Delta \quad &\text{otherwise}
\end{dcases}
\]
as the radius of the smallest ball centered at $\yb^0$ that contains $Y$.
When $Y$ contains (potentially copies of) a single unique point, we take $\Delta(Y) = \Delta$ by default, where $\Delta$ is a trust-region radius as discussed in \Cref{subsec:Birkhoff_poisedness}.
Entries $M_{ij}$ of $\Mb$ are scaled to different orders of magnitude in $\Delta(Y)$, corresponding to the total derivative degree $|\alphab^i|$ of the associated interpolation condition and the total polynomial degree of the associated polynomial basis function $\phi_j$.
In particular, the condition number of $\Mb$ depends on $\Delta(Y)$, and the system \eqref{eq:m_interp_linsys} can be ill-conditioned when $\Delta(Y)$ is many orders of magnitude away from 1.
For this reason, we scale the rows and columns %
of $\Mb$ to obtain a normalized matrix $\hat{\Mb} = \hat{\Mb}(\phi, D)$ by the transformation
\begin{equation} \label{eq:Mhat}
    \hat{\Mb} \defined
    \begin{bmatrix}
        \Delta(Y)^{|\alphab^0|} & & \\
        & \ddots & \\
        & & \Delta(Y)^{|\alphab^{q}|}
    \end{bmatrix}
    \Mb
    \begin{bmatrix}
        1 & & \\
        & \Delta(Y)^{-1} \Ib_{n} & \\
        & & \Delta(Y)^{-2} \Ib_{n(n+1)/2}
    \end{bmatrix}
    .
\end{equation}
The nonzero entries
\[
\hat{M}_{i j} = \partial^{\alphab^i} \phi_j \left( \frac{\yb^i - \yb^0}{\Delta(Y)} \right)
\]
are of order 1, and the interpolation conditions \eqref{eq:m_interp_linsys} can be expressed by the equivalent linear system
\begin{equation} \label{eq:m_interp_linsys_normalized}
    \hat{\Mb}
    \begin{bmatrix}
        c \\ \Delta(Y) \gb \\ \Delta(Y)^2 \textvec{\Hb}
    \end{bmatrix}
    =
    \begin{bmatrix}
        \Delta(Y)^{|\alphab^0|} \partial^{\alphab^0} f(\yb^0) \\
        \vdots \\
        \Delta(Y)^{|\alphab^{q}|} \partial^{\alphab^{q}} f(\yb^{q})
    \end{bmatrix}
    .
\end{equation}

When building a Birkhoff interpolation model using points $Y$, we always center the model at $\yb^0$ and scale $Y$ to obtain the normalized points
\begin{equation} \label{eq:rescale_Y}
    \hat{Y} = \left\{ \hat{\yb}^0, \hat{\yb}^1, \dots, \hat{\yb}^{q} \right\}
    = \left\{ \zerob, \frac{\yb^1 - \yb^0}{\Delta(Y)}, \dots, \frac{\yb^q - \yb^0}{\Delta(Y)} \right\}
    \subset \Ball(\zerob, 1).
\end{equation}
In this manuscript, we will consistently use the hat accent to denote this normalization procedure.
We note that the normalized matrix $\hat{\Mb}$ can be equivalently expressed as the interpolation matrix $\Mb$ obtained from centering and scaling the points in $D$:
\[
\hat{\Mb}(\phi, \{ (\yb^i, \alphab^i) \}) =
\Mb(\phi, \{ (\hat{\yb}^i, \alphab^i) \} )
.
\]
We tested this rescaling numerically. When $D$ has a 
matching-order derivative condition for every entry of $(c,\gb,\Hb)$,
$\text{cond}(\Mb)$ is insensitive to $\Delta(Y)$ and rescaling makes little
difference. However, in our experiments when entries of $\Hb$ are recovered from
function values, we observed $\text{cond}(\Mb)$ to grow like
$\Delta(Y)^{-4}$ as $\Delta(Y) \to 0$ while $\text{cond}(\hat{\Mb})$ stayed
bounded. We keep the rescaling for this reason, and because it keeps
the interpolation matrix entries at a uniform scale.

\section{Algorithm Description} \label{sec:alg}

We begin by developing routines for use in a convergent trust-region optimization algorithm that employs Birkhoff interpolation models.
Toward provable guarantees on approximation quality,
we first require a notion of \emph{well-poised} interpolation data $D$, which will guide our model-building procedure.
After providing this definition of well-poisedness, we will present and analyze methods for maintaining well-poised interpolation data $D$ over the course of running an optimization method.
Since we are extending the theory of~\cite{CSV2009}, we will point to the
corresponding analogues as we proceed with the following analysis.

\subsection{Birkhoff Poisedness} \label{subsec:Birkhoff_poisedness}

The concept of poisedness measures the geometric quality of interpolation data and plays a central role in establishing bounds on model accuracy.
Informally, well-poised interpolation sets avoid nearly linearly dependent interpolation conditions and ensure that local polynomial models remain stable under perturbations in the data.

We first develop a notion of well-poisedness of Birkhoff interpolation data $D$ analogous to the notion for Lagrange interpolation presented in, for example,~\cite{CSV2009}.

When $\Mb$, or equivalently $\hat{\Mb}$, is invertible, the coefficients of the
Birkhoff interpolant $m$ are uniquely defined.
In this case, we say
that the data $D$
describes a \emph{poised} Birkhoff interpolation set.

\begin{definition}
    Given data $D = \{ (\yb^i , \alphab^i) \mid i = 0, \dots, q \}$,
    we say $D$ is \emph{poised for Birkhoff interpolation} if the associated matrix $\Mb(\phi, D)$ is invertible or, equivalently, if the normalized matrix $\hat{\Mb}(\phi, D)$ is invertible.
\end{definition}

Toward a definition of well-poised Birkhoff interpolation data, we define a set of Birkhoff interpolation polynomials, which are an obvious analogue to Lagrange interpolation polynomials.

\begin{definition}
    Given a set $D$ of $q + 1$ interpolation conditions,
    the associated Birkhoff interpolation polynomials
    \[
    \lambda_i(\xb), \quad i = 0, \dots, q
    \]
    are the quadratic polynomials that satisfy
    \begin{equation} \label{eq:HIP_conditions}
        \partial^{\alphab^{i'}} \lambda_i(\hat{\yb}^{i'}) = \delta_{i i'}
    \end{equation}
    for all $i, i' = 0, \dots, q$.
    That is, the $i$th Birkhoff interpolation polynomial, when differentiated by the multi-index $\alphab^i$ and evaluated at the corresponding normalized point $\hat{\yb}^i = (\yb^i - \yb^0)/\Delta(Y)$, evaluates to $1$;
    at all other $(\hat{\yb}^{i'}, \alphab^{i'})\in D$, the corresponding derivative of $\lambda^i$ evaluates to $0$.
\end{definition}

We emphasize that the conditions describing the Birkhoff interpolation polynomials \eqref{eq:HIP_conditions} require normalizing the points $Y$ to the unit ball (in the $\alphab \neq \zerob$ case).
Without normalization, the Birkhoff interpolation polynomials are scale-dependent in the sense that $\partial^{\alphab} [\lambda_i(\xb/\Delta(Y))] = \Delta(Y)^{-|\alphab|} \partial^{\alphab}\lambda_i(\xb/\Delta(Y))$ by the chain rule.
The consequences of rescaling become apparent in the sequel to this discussion in \cref{def:Lambda_poised}.

In the derivative-free case where $A = \{\zerob\}$, Birkhoff interpolation polynomials are simply Lagrange interpolation polynomials.
The Birkhoff interpolation polynomials retain properties similar to those of the Lagrange interpolation polynomials.

The primary distinction from the classical Lagrange interpolation setting is that interpolation conditions are no longer associated solely with function values.
Consequently, the geometry of the interpolation set depends jointly on sample locations and derivative-selection patterns.
Despite this added flexibility, the resulting stability bounds retain essentially the same structure as in the classical derivative-free case.

These similarities are stated in the following two theorems.
The first is an analog of~\cite[Lemma 3.4]{CSV2009}; the second is an analog of~\cite[Lemma 3.5]{CSV2009}.

\begin{lemma} \label{thm:BIPs_unique}
    When $D$ is poised for Birkhoff interpolation, the Birkhoff interpolation polynomials exist and are unique.
\end{lemma}

\begin{proof}
    Since $D$ is poised, the matrix $\Mb(\phi, \{ (\hat{\yb}^i, \alphab^i) \})$ in \eqref{eq:m_interp_linsys} is invertible, so for each $0 \le i \le q$ the system defining the coefficients of $\lambda_i$ has a unique solution.
\end{proof}

\begin{theorem} \label{thm:m_lincomb_Birkhoff_polys}
    The interpolating
    polynomial $m(\xb)$ satisfying the poised conditions
    \eqref{eq:m_interp} may be expressed in terms of the Birkhoff
    interpolation polynomials as
    \begin{equation} \label{eq:m_from_BIPs}
        m(\xb) =
        \sum_{i=0}^{q} \Delta(Y)^{|\alphab^i|} \partial^{\alphab^i} f(\yb^i)
        \lambda_{i} \left(\frac{\xb - \yb^0}{\Delta(Y)}\right) .
    \end{equation}
\end{theorem}

\begin{proof}
  Differentiating \eqref{eq:m_from_BIPs} according to $\alphab^{i'}$ and evaluating at $\yb^{i'}$, the chain rule gives
  \begin{align*}
      \partial^{\alphab^{i'}} m(\yb^{i'})
        &= \sum_{i=0}^{q} \Delta(Y)^{|\alphab^i| - |\alphab^{i'}|} \partial^{\alphab^i} f(\yb^i)
        \partial^{\alphab^{i'}} \lambda_{i} \left( \hat{\yb}^{i'} \right) \\
      &= \sum_{i=0}^{q} \Delta(Y)^{|\alphab^i| - |\alphab^{i'}|} \partial^{\alphab^i} f(\yb^i)
        \delta_{i i'} = \partial^{\alphab^{i'}} f(\yb^{i'}).
  \end{align*}
  This holds for all $i' = 0, \dots, q$, which are precisely the conditions \eqref{eq:m_interp}.
\end{proof}

Birkhoff interpolation polynomials also induce a \emph{poisedness constant}
based on upper bounds on the values of Birkhoff interpolating polynomials over a domain; this constant is analogous to the poisedness constant from Lagrange interpolation.

\begin{definition} \label{def:Lambda_poised}
    We say $D$ is $\Lambda$-poised with respect to a multi-index set $A$ in a region $\cX \subset \Reals^n$ if and only if
    \[
    \left| \partial^{\alphab} \lambda_i \left( \frac{\xb - \yb^0}{\Delta(Y)} \right) \right| \le \Lambda
    \quad
    \]
    uniformly
    for all $i = 0, \dots, q$,
    for all $\alphab \in A$,
    and for all $\xb \in \cX$.
\end{definition}

    \Cref{def:Lambda_poised} of the Birkhoff poisedness constant coincides with the definition of the Lagrange poisedness constant presented in, for example,~\cite{CSV2009}.
    To see this, let $A=\{\zerob\}$, and let the points $Y=\{\yb^0, \dots, \yb^q\}$ be given and poised for Lagrange interpolation.
    Let $\lambda_i^L$ be the Lagrange interpolating polynomials satisfying $\lambda_i^L(\yb^{i'}) = \delta_{i i'}$, and let $\lambda_i^B$ be the Birkhoff interpolating polynomials satisfying $\lambda_i^B((\yb^{i'} - \yb^0)/\Delta(Y)) = \delta_{i i'}$.
    By \Cref{thm:BIPs_unique} and~\cite[Lemma 3.4]{CSV2009}, both sets of polynomials $\{\lambda_i^L\}$ and $\{\lambda_i^B\}$ are unique.
    Moreover, they are related by the identity
    $\lambda_i^L(\xb) = \lambda_i^B((\xb - \yb^0)/\Delta(Y))$ for each $i$.
    It immediately follows that
    \[
    \max_{0 \le i \le q} \max_{\xb \in \cX}
    \abs{\lambda_i^L(\xb)} \le \Lambda
    \quad \iff \quad
    \max_{0 \le i \le q} \max_{\xb \in \cX}
    \abs{\lambda_i^B\left( \frac{\xb - \yb^0}{\Delta(Y)}\right)} \le \Lambda.
    \]
    In other words, $Y$ is $\Lambda$-poised in the Lagrange sense if and only if $D = \{(\yb^i, \zerob)\}$ is $\Lambda$-poised in the Birkhoff sense with respect to $A=\{ \zerob \}$.

    In the context of a trust-region framework, we will always set $\cX = \Ball(\yb^0, \Delta)$ in \cref{def:Lambda_poised}, where $\Delta$ is some trust-region radius.
    When $Y$ is strictly contained in $\Ball(\yb^0, \Delta)$, the scaling constant appearing in \cref{def:Lambda_poised} is specifically $\Delta(Y)$, not $\Delta$.

\Cref{fig:heatmap_poisedness} shows the $\Lambda$-poisedness after adding interpolation
information in two possible unit trust regions,
centered at $(0,0)$. In both panels, the initial interpolation data is $D = \{(\yb^0, [0,0]), (\yb^0, [1,0]), (\yb^1, [0,0]), (\yb^1, [1,0])\}$, where $\yb^0 = (0,0)$ and $\yb^1 = (\sqrt{2}/2,\sqrt{2}/2)$; that is, both $\yb^0$ and $\yb^1$ already have a function-value condition and a $\partial_{x_1}$-derivative condition fixed in $D$, indicated by the square markers in \Cref{fig:heatmap_poisedness}.
Each heatmap plots $\Lambda$ as a function of a candidate interpolation point $(x_1,x_2)$ that would be added to $D$, ranging over the horizontal and vertical axes of the panel. The left panel illustrates the value of $\Lambda$ that results when we add to $D$ the two data $([x_1,x_2], [0,0])$ and $([x_1,x_2], [1,0])$ at that candidate point.
The right panel first augments the initial interpolation set $D$ with an additional, fixed function-value condition $((\sqrt{2}/2,-\sqrt{2}/2), [0,0])$, indicated by the circle marker in \Cref{fig:heatmap_poisedness}, and then displays the resulting $\Lambda$ values when we add to $D$ only the function-value datum $([x_1,x_2], [0,0])$ at the candidate point. %
Darker
regions correspond to smaller values of $\Lambda$ and therefore more
well-poised sets.
These heatmaps resemble Fig.~11.1 of~\cite{CSV2009}, which depicts an analogous ``wedge'' method for Lagrange interpolation.

These plots show how the quality of a Birkhoff interpolation set depends
not only on the spatial distribution of interpolation points but also on the
distribution of derivative information across those points. In particular,
certain combinations of function and derivative conditions can produce considerably different poisedness constants, even when the underlying point set is
similar.

\begin{figure}
    \centering
    \includegraphics[width=0.49\linewidth]{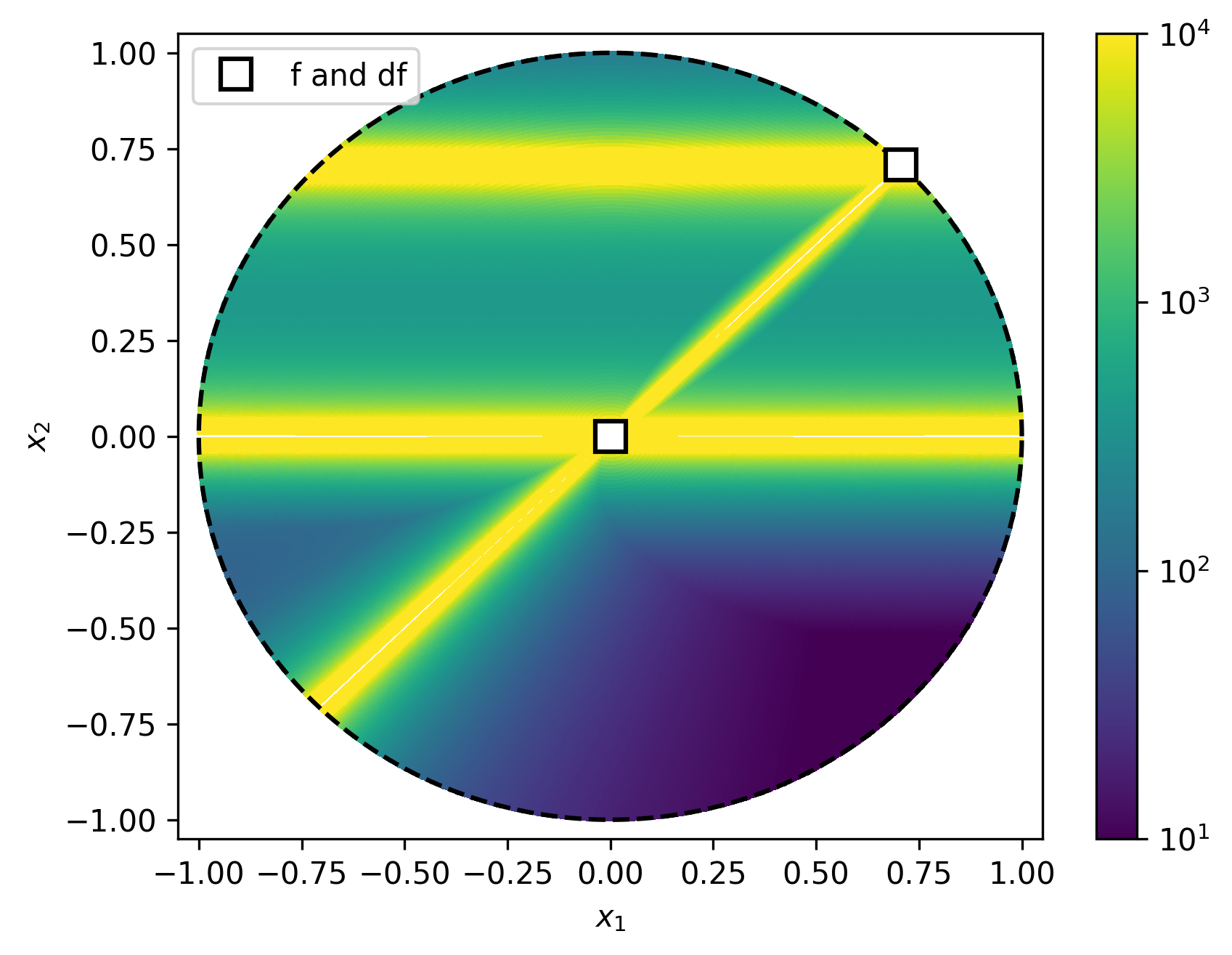}
    \includegraphics[width=0.49\linewidth]{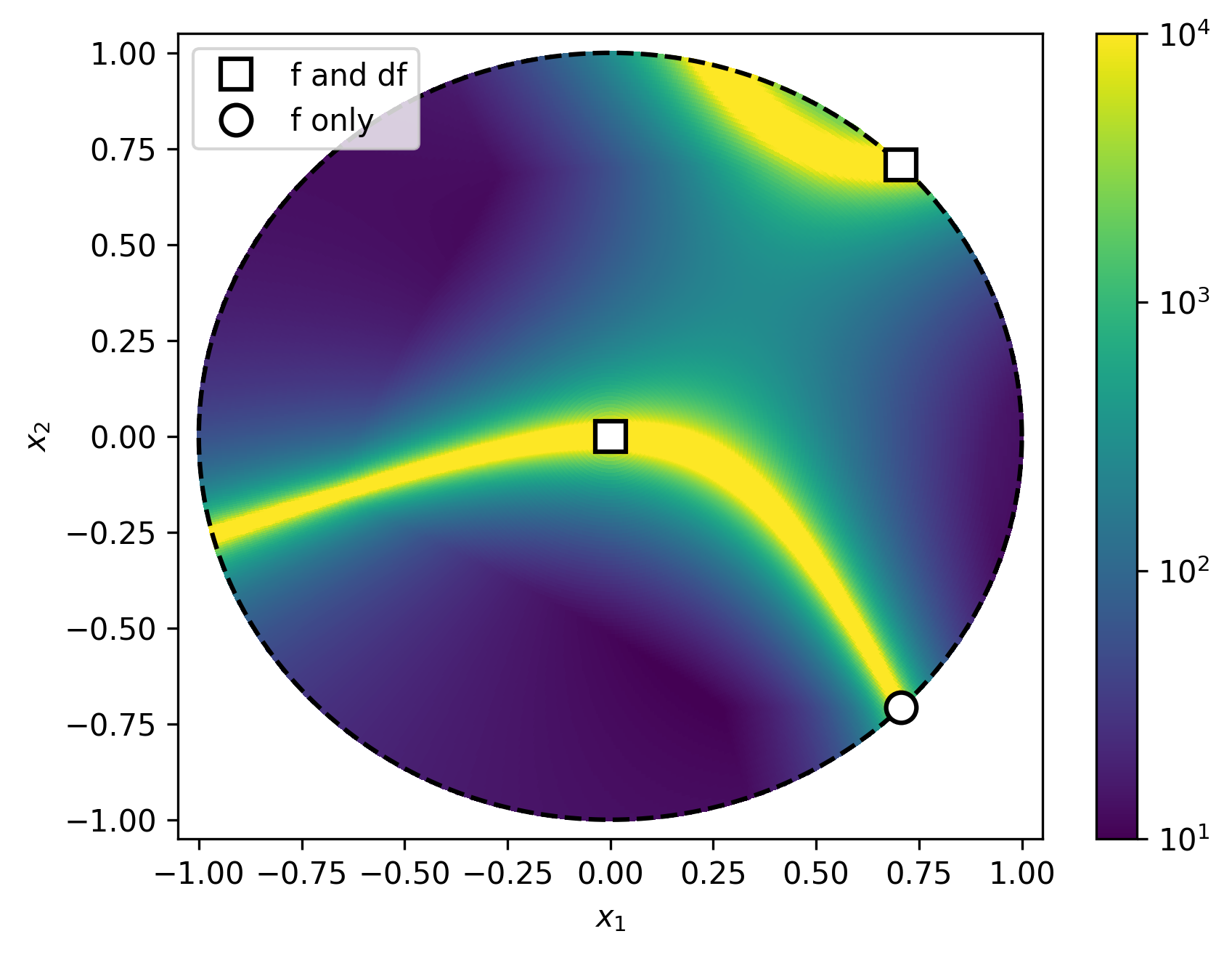}
    \caption{
Heatmaps of the Birkhoff poisedness constant $\Lambda$ for two interpolation
configurations in $\Ball(\zerob,1)$, plotted as a function of a candidate interpolation point $(x_1,x_2)$. Smaller values are better poised. Square markers indicate points already in $D$ with both a function-value and a $\partial_{x_1}$-derivative condition fixed; the circle marker (right panel only) indicates a point already in $D$ with only a function-value condition fixed. The left panel
adds both function and derivative interpolation conditions at the candidate
point, while the right panel adds only a function-value condition after fixing
an additional interpolation point.}
    \label{fig:heatmap_poisedness}
\end{figure}

Analogous to Lagrange interpolation polynomials, an equivalent definition
of the poisedness constant $\Lambda$ can be derived from a linear algebra perspective.
Given poised data $D$,
we can conclude from the uniqueness of the interpolating polynomial $m(\xb)$ guaranteed by \Cref{thm:BIPs_unique},
and the form of $m(\xb)$ given in
\Cref{thm:m_lincomb_Birkhoff_polys}, that
\[
\phi_j \left( \hat{\xb} \right) =
\sum_{i=0}^{q} %
\partial^{\alphab^i} \phi_j\left( \hat{\yb}^i \right)
\lambda_{i} \left( \hat{\xb} \right)
\]
for all $j = 0, \dots, q$.
Equivalently, in matrix notation,
\begin{equation} \label{eq:lambda_linsys_normalized}
    \hat{\Mb}^T
    \begin{bmatrix}
        \lambda_0( \hat{\xb} ) \\
        \vdots \\
        \lambda_q( \hat{\xb} )
    \end{bmatrix}
    =
    \phib(\hat{\xb})
    .
\end{equation}
Hence from \Cref{def:Lambda_poised}, we conclude that
$D$ is $\Lambda$-poised with respect to $A$ in $\cX$ if and only if
for all $\xb \in \cX$ there exists a vector $\lambdab(\hat{\xb})$
such that $\hat{\Mb}^T \lambdab(\hat{\xb}) = \phib(\hat{\xb})$
and $\max_{\alphab \in A} \normil{\partial^{\alphab} \lambdab(\hat{\xb})}_{\infty} \le \Lambda$.
From this equivalent form, we may relate the poisedness constant
$\Lambda$ to $\normil{\hat{\Mb}^{-1}}$.
For this, we state a direct translation of~\cite[Lemma 3.10]{CSV2009}.

\begin{lemma} \label{lem:CSV3.10}
    There exist a number $\sigma_{\infty} > 0$ such that, for any choice of
    $\vb \in \Reals^{q+1}$ satisfying $\norm{\vb}_{\infty} = 1$,
    there exists a $\hat{\xb} \in \Ball(\zerob, 1)$ such that $\abs{\vb^T \phib(\hat{\xb})}
    \ge \sigma_{\infty}$.
\end{lemma}
The constant $\sigma_{\infty}$ in \Cref{lem:CSV3.10} depends only on the choice of basis.
In particular, the natural basis of monomials of total degree at most 2 has $\sigma_{\infty} = 1/4$.
When $\norm{\vb}_{\infty} \ne 1$ and $\vb \ne \zerob$, one can apply this lemma to $\vb / \norm{\vb}_{\infty}$ to conclude that
\[
\max_{\hat{\xb} \in \Ball(\zerob, 1)}
\abs{\vb^T \phib(\hat{\xb})} \ge \norm{\vb}_{\infty} \sigma_{\infty}
.
\]
This inequality also holds trivially when $\vb = \zerob$, since both sides vanish.

We now show explicitly how $\Lambda$ is related to $\normil{\hat{\Mb}^{-1}}$, analogous to~\cite[Theorem 3.14]{CSV2009}.
In practice,
$\normil{\hat{\Mb}^{-1}}$
serves as a numerical conditioning measure for the interpolation system. Large values indicate that the interpolation conditions are close to linearly dependent, which may lead to unstable model coefficients. %

\begin{theorem} \label{thm:Lambda_iff_invnorm}
    Let \Cref{assn:A_D} hold.
    If $\normil{\hat{\Mb}^{-1}} \le \Lambda/\sqrt{q+1}$, then $D$ is $\Lambda$-poised for Birkhoff interpolation with respect to
    $A$ in $\Ball \defined \Ball(\yb^0, \Delta(Y))$.
    Conversely, if $D$ is $\Lambda$-poised for Birkhoff interpolation with
    respect to $A$ in $\Ball$, then $\normil{\hat{\Mb}^{-1}} \le C \Lambda$
    for a constant $C = C(n, \phi)$ that depends only on the problem dimension and the choice of the natural basis.
\end{theorem}

\begin{proof}
    Throughout the proof, we use the fact that $\normil{\hat{\Mb}^{-1}} = \normil{\hat{\Mb}^{-T}}$.
    To prove the forward relation, let $(\xb, \alphab) \in \Ball \times A$ be given.
    From the system $\hat{\Mb}^T \lambdab(\hat{\xb}) = \phib(\hat{\xb})$,
    \begin{align*}
        \lambda_i(\hat{\xb})
            = \sum_{j=0}^{q} (\hat{\Mb}^{-T})_{ij} \phi_j(\hat{\xb}) 
        &\implies \partial^{\alphab} \lambda_i(\hat{\xb})
            = \sum_{j=0}^{q} (\hat{\Mb}^{-T})_{ij}
            \partial^{\alphab} \phi_j(\hat{\xb}) \\
        &\implies |\partial^{\alphab} \lambda_i(\hat{\xb})|
            \le \sum_{j=0}^{q} |(\hat{\Mb}^{-T})_{ij}|
            |\partial^{\alphab} \phi_j(\hat{\xb})| \le \norm{\hat{\Mb}^{-T}}_\infty
    \end{align*}
    since $|\partial^{\alphab} \phi_j(\hat{\xb})| \le 1$ for all $\hat{\xb} \in \Ball(\zerob, 1)$ and all $0 \le j \le q$ and all $\alphab \in A$.
    Since we have $\normil{ \hat{\Mb}^{-T} }_\infty \le \sqrt{q+1} \normil{\hat{\Mb}^{-T}}_2 \le \Lambda$,
    this proves the forward relation.

    To prove the reverse relation, %
    recall the vector space $\cP_n^2$ of quadratic polynomials with total degree at most two.
    For $p \in \cP_n^2$, consider the quantity
    \[
    \norm{p}_A \defined \max_{\alphab \in A} \max_{\xb \in \Ball} |\partial^{\alphab} p(\hat{\xb})|
    .
    \]
    Observe that $\norm{\cdot}_A$ is a norm on $\cP_n^2$.
    To see that $\norm{\cdot}_A$ is a norm, note that
    $\normil{p}_A = 0$ if and only if $p$ is the zero polynomial.
    It is apparent that $\norm{b p}_A = |b| \norm{p}_A$ for $b \in \Reals$.
    Finally, $\norm{\cdot}_A$ satisfies the triangle inequality because
    \begin{align*}
        \norm{p_1 + p_2}_A
        &= \max_{\alphab \in A} \max_{\xb \in \Ball} |\partial^{\alphab} p_1(\hat{\xb}) + \partial^{\alphab} p_2(\hat{\xb})| \\
        &\le \max_{\alphab \in A} \max_{\xb \in \Ball} \left( |\partial^{\alphab} p_1(\hat{\xb})|
            + |\partial^{\alphab} p_2(\hat{\xb})| \right) \\
        &\le \max_{\alphab \in A} \max_{\xb \in \Ball} |\partial^{\alphab} p_1(\hat{\xb})|
            + \max_{\betab \in A} \max_{\yb \in \Ball} |\partial^{\betab} p_2(\hat{\yb})| \\
        &= \norm{p_1}_A + \norm{p_2}_A
        .
    \end{align*}

    We now define a second norm for $p \in \cP_n^2$,
    \[
    \norm{p}_{\text{coeff}} \defined \norm{\vb} \text{ for the unique $\vb$ such that } p(\hat{\xb}) = \vb^T \phib(\hat{\xb}).
    \]
    As a second norm on the finite-dimensional space $\cP_n^2$, $\normil{\cdot}_{\text{coeff}}$ is equivalent to $\normil{\cdot}_A$; that is, there exists a constant $C' > 0$ for which
    $\norm{p}_{\text{coeff}} \le C' \norm{p}_A$
    for all $p \in \cP_n^2$.
    Specifically, we can take
    $C' = \sup_p \{ \norm{p}_{\text{coeff}} :\norm{p}_A = 1 \}$
    .
    When $\norm{p}_A = 1$, \cref{lem:CSV3.10} and $\zerob \in A$ (from \Cref{assn:A_D}) imply
    \[
    \norm{\vb}_{\infty} \sigma_{\infty}
    \le \max_{\xb \in \Ball} \abs{\vb^T \phib(\hat{\xb})}
    = \max_{\xb \in \Ball} \abs{p(\hat{\xb})}
    \le \norm{p}_A
    = 1.
    \]
    It follows that $\norm{\vb} \le \sqrt{q+1}/\sigma_{\infty}$, and so we may use
    $C' = \sqrt{q+1}/\sigma_{\infty}$.

    Let $\vb^i$ be the representation of $\lambda_i$ in the natural basis, $\lambda_i(\hat{\xb}) = (\vb^i)^T \phib(\hat{\xb})$.
    Then
    \[
    \hat{\Mb}^{-1} = [ \vb^0, \dots, \vb^{q}]
    \]
    by the orthonormality of $\{\lambda_i\}$,
    and so
    \begin{align*}
        &\norm{\hat{\Mb}^{-1}} = \sup_{\norm{\zb}=1} \norm{z_0 \vb^0 + \dots + z_{q} \vb^{q}}
        \le \sup_{\norm{\zb}=1} \sum_{i=0}^{q} |z_i| \norm{\vb^i} \\
        &\le \sup_{\norm{\zb}=1} \norm{\zb} \sqrt{\sum_{i=0}^{q} \norm{\vb^i}^2}
        = \sqrt{\sum_{i=0}^{q} \norm{\lambda_i}_{\text{coef}}^2}
        \le \frac{\sqrt{q + 1}}{\sigma_{\infty}} \sqrt{ \sum_{i=0}^{q} \norm{\lambda_i}_A^2 }
        \le \left( \frac{q + 1}{\sigma_{\infty}} \right) \Lambda.
    \end{align*}
    Taking $C(n, \phi) = (q + 1) / \sigma_{\infty}$, we have shown the desired result.
\end{proof}

From this, we immediately have the following corollary, which bounds the volume associated with the matrix $\hat{\Mb}$ uniformly away from zero.
\begin{corollary} \label{cor:volume_lower_bound}
    Let \Cref{assn:A_D} hold.
    If $D$ is $\Lambda$-poised in $\Ball(\yb^0, \Delta(Y))$ with respect to $A$, then $|\det{\hat{\Mb}(\phi, D)}|$ is bounded below, with
    \[
    \abs{\det(\hat{\Mb}(\phi, D))} \ge \Theta(n, \phi, \Lambda) > 0.
    \]
\end{corollary}
\begin{proof}
    The previous theorem guarantees a constant $C(n, \phi)$ such that if $D$ is $\Lambda$-poised in $\Ball(\yb, \Delta(Y))$ with respect to $A$, then $\normil{\hat{\Mb}^{-1}} \le C(n, \phi) \Lambda$.
    Since the absolute value of a matrix determinant is the product of its singular values, we obtain
    \[
    \abs{\det(\hat{\Mb})}
    = \frac{1}{|\det(\hat{\Mb}^{-1})|}
    \ge \frac{1}{ (C(n, \phi) \Lambda)^{q+1} }
    .
    \]
\end{proof}

\subsection{Pivoting Algorithms for Improving Poisedness}

In this section we describe subroutines for generating poised interpolation data and improving the poisedness constant of existing data.
These subroutines are greedy and prioritize computational efficiency.
Nonetheless, we will be able to demonstrate theoretical guarantees concerning these subroutines, which will be necessary for proving convergence results for our model-based DFO algorithm.

\Cref{alg:model_completion} describes a procedure for generating poised Birkhoff interpolation data from non-poised data.
In the context of our model-based DFO algorithm, at the outset of an iteration there may not exist a set of poised data within a reasonable distance of the current trust region, necessitating \Cref{alg:model_completion}.
As a special extreme case, \Cref{alg:model_completion} can be employed to generate an initial set of interpolation data given a single $(\yb, \alphab)$ pair.

\Cref{alg:model_completion} does not compute the Birkhoff interpolation polynomials explicitly.
Instead, we develop a basis of \emph{pivot polynomials} $\{u_i \mid i=0, \dots, q\}$ that span the space $\cP_n^2$.
These pivot polynomials do not satisfy the mutual orthonormality of Birkhoff interpolation polynomials, but they do satisfy $\partial^{\alphab^i} u_j(\hat{\yb}^i) = \delta_{ij}$ for $i \le j$.
This relaxation reduces the computational load in each iteration of \Cref{alg:model_completion}, but it  introduces a greedy characteristic; pivot polynomials constructed in early iterations of \Cref{alg:model_completion} are never re-examined in later iterations.

To clarify notation,
in \Cref{alg:model_completion} we continue to use $(\yb,\alphab)$ to denote interpolation conditions in the input data $D$, but we will introduce $(\zb,\betab)$ to denote interpolation conditions in a secondary data set $D'$, which will be returned upon termination.
We write $p+1$ to denote the cardinality of $D$, which may or may not equal $q+1$, the number of conditions in the poised set $D'$. If $p<q$, conditions will be added until the result is poised; if $p>q$, conditions will be removed.

\begin{algorithm2e}[h!]
    \SetAlgoNlRelativeSize{-4}
\caption{Birkhoff model completion with pivot polynomials}
\label{alg:model_completion}
\KwIn{Polynomial basis $\{ u_i \mid i=0,1, \dots,q \}$ (e.g., the natural basis $\phi$).
Model center $\yb^0$.
Interpolation data $D = \{ (\yb^i, \alphab^i) \mid i=0, \dots, p \}$
with $(\yb^0, \zerob) \in D$.
Radius $\Delta$.
Acceptance threshold $\xi_{acc} > 0$.\\ %
}

\KwOut{Data $D'$ numerically poised at least to threshold $\xi_{acc}$ or a certificate that $\xi_{acc}$ is too high. Final pivot polynomial $u_q$ for which $\partial^{\alphab^i} u_q(\hat{\yb^i}) = 0$ for $i=0, \dots, q-1$.}

Initialize $D' \gets \{ (\zb^0 := \yb^0, \betab^0 := \zerob) \}$
and $D \gets D \setminus \{ (\yb^0, \zerob) \}$.\\

$D \gets D \setminus \left\{ (\yb^i, \alphab^i) \in D : \norm{\yb^i - \yb^0} > \Delta  \right\}$.

\For{$i=1, \dots, q$}{

    $\displaystyle
    (\zb^i, \betab^i) \gets \argmax_{(\yb, \alphab) \in D} \abs{ \partial^{\alphab} u_i(\hat{\yb}) }
    $
    with the rescaling for $\hat{\yb}$ defined in \eqref{eq:rescale_Y}.\\
    $\displaystyle M_{i, D} \gets \abs{ \partial^{\betab^i} u_i(\hat{\zb}^i) }$.\\
    \uIf{$M_{i, D} > \xi_{acc}$}{

        $D \gets D \setminus \{(\zb^i, \betab^i) \}$.

    }
    \Else{

        $\displaystyle
        (\zb^i, \betab^i) \gets \argmax_{(\yb, \alphab) \in \Ball(\yb^0, \Delta) \times A}
        \abs{\partial^{\alphab} u_i(\hat{\yb})}
        $.\\

        $\displaystyle
        M_{i, \Delta} \gets \abs{ \partial^{\betab^i} u_i(\hat{\zb}^i) }
        $.\\

        \If{$M_{i, \Delta} < \xi_{acc}$}{
            \Return \textbf{failure}
        }

    }

    $\displaystyle D' \gets D' \cup \{ (\zb^i, \betab^i) \}$.\\

    \For{$j = i + 1, \dots, q$}{
        $\displaystyle
        u_j(\hat{\xb}) \gets u_j(\hat{\xb}) - \frac{\partial^{\betab^i} u_j(\hat{\zb}^i)}{\partial^{\betab^i} u_i(\hat{\zb}^i)} u_i(\hat{\xb}).
        $ \label{line:gauss_elim}
    }
}
\end{algorithm2e}

A key challenge we face, and as witnessed in \Cref{fig:heatmap_poisedness}, is that derivative information may improve interpolation quality unevenly across candidate conditions. Since including all available derivative information may unnecessarily increase interpolation-system size and ill-conditioning, we seek a strategy that selectively incorporates the most useful derivative constraints. The greedy procedure in \Cref{alg:model_completion} attempts to balance interpolation quality with model complexity.

In \Cref{alg:model_completion} we explicitly initialize the poised interpolation set with the condition $m(\yb^0) = f(\yb^0)$.
Practically, the satisfaction of this constraint is guaranteed by choosing $(\yb^0, \zerob)$ as the first pivot on the initial (natural basis) pivot polynomial $u_0(\xb) = 1$.
If we do not make this initial choice of pivot, the condition $(\yb^0, \zerob)$ may not pivot into $D'$ in later iterations.
Although there is nothing incorrect with omitting the condition $m(\yb^0) = f(\yb^0)$ from the Birkhoff interpolation model, we find that its inclusion makes models more intuitive and interpretable.

We also note the role of scaling and weighting in \Cref{alg:model_completion}.
In every iteration, the normalizing factor that transforms $\yb$ into $\hat{\yb}$ and $\zb$ into $\hat{\zb}$ is fixed to be the maximal radius $\Delta(Y)$ associated with the inputted data $D$.
If we recalculated the scale factors for $D$ and $D'$ in each iteration, then comparing pivot values to the constant threshold $\xi_{acc}$ would also change and cause ambiguity about the quality of the returned data.
On the other hand, there is a different form of rescaling that maintains consistency throughout the algorithm.
One can introduce a weighting function $w(\yb, \alphab)$ and replace all derivatives $\partial^{\alphab} u_i(\hat{\yb})$ by a weighted variant $\partial^{\alphab} u_i(\hat{\yb}) / w(\yb, \alphab)$.
This can be used, for example, to prioritize selecting interpolation nodes that lie within the current trust region.
As long as $w(\yb, \alphab)$ is positive and uniformly bounded above on the entire evaluation domain $\cL'(\xb^{(0)}) \times A$, \Cref{alg:model_completion} will still produce a poised interpolation set, and the following theorem guaranteeing convergence still applies.

\begin{theorem}\label{thm:alg1terminates}
    Let \Cref{assn:A_D} hold, and let interpolation conditions $D$ be given, with $(\yb^0, \zerob) \in D$.
    \Cref{alg:model_completion} will successfully terminate with poised Birkhoff interpolation data as long as $\xi_{acc}$ is sufficiently small.
\end{theorem}

\begin{proof}
    Let $\{ u_j^{(i)} \mid j=0, \dots, q \}$ denote the polynomial basis spanning $\cP_2^n$ at the beginning of the $i$th iteration of \Cref{alg:model_completion}, where $\{ u_j^{(0)} \}$ denotes the input basis.
    None of $u_0^{(0)}, \dots, u_q^{(0)}$ can be the zero polynomial, since they form a basis for the polynomial space.
    At the outset of the $i$th iteration, the polynomial $u_i^{(i)}$ is a linear combination of $u_0^{(0)}, \dots, u_i^{(0)}$, due to \lineofalg{line:gauss_elim}{alg:model_completion}.
    Notice that, because of  \lineofalg{line:gauss_elim}{alg:model_completion}, no pivot polynomial $u_j^{(i)}$ can ever be the zero polynomial, since
    $a_{ii}=1$ in
    \[
    u_i^{(i)}(\hat{\xb}) = \sum_{j=0}^i a_{ji} u_j^{(0)}(\hat{\xb}).
    \]
    Thus, for any $\Delta > 0$ there always exists some $\yb \in \Ball(\zerob, \Delta)$ such that $\abs{u_i(\hat{\yb})} > 0$.

    Next, consider the quantities
    $M_{i, D}$ and $M_{i, \Delta}$, which appear in the $i$th iteration of \Cref{alg:model_completion}, and define
    \begin{align*}
        M_i &\defined
            \begin{dcases}
                M_{i, D} & M_{i, D} \ne 0, \\
                M_{i, \Delta} & M_{i, D} = 0.
            \end{dcases}
    \end{align*}
    If $M_{i, D} = 0$, then
    $M_{i, \Delta} > 0$ by our previous observation that $u_i^{(i)}$ is not the zero polynomial.
    Since $\zerob \in A$ by \Cref{assn:A_D}, this guarantees $M_i > 0$ at each iteration.
    Thus, \Cref{alg:model_completion} will terminate with success provided $\xi_{acc} < \min_i \{ M_i \}$.
\end{proof}

\Cref{alg:model_completion} (and \Cref{alg:model_improvement} and \Cref{alg:optimizer} to follow)  assumes knowledge only of the available derivative set $A$.
Critically, one does not need to know {\em a priori} which derivatives should be utilized to build a Birkhoff model.
In fact, fixing an a priori ordering on multiindices $\alphab$ drawn from the availability set $A$ could raise potential issues.
Consider a modification of \Cref{alg:model_completion} in two dimensions given input $\bar{D} = \{ (\yb^0 = \zerob, \zerob) \}$ with the natural quadratic basis as the initial pivot polynomials.
If one insists on building a model with three conditions with $\alphab = \zerob$ and three conditions with $\alphab = [1,0]$, then the success of \Cref{alg:model_completion} depends on the orderings of both the pivot basis functions and the derivative multi-indices.
For example, consider providing input $D = \{ (\yb^0, \zerob) \}$ to \Cref{alg:model_completion}.
In the first pass through the for loop of \Cref{alg:model_completion},
we deal with the pivot polynomial $u_1(\xb) = \xb_1$.
Suppose we had fixed the multi-index $\alphab = [1,0]$ in this first pass.
Since $\partial_{\xb_1}u_1(\xb) = 1$, we may arbitrarily choose $\yb^1 = \yb^0 = \zerob$.
After the corresponding orthogonalization in this first pass, the selection of
neither $(\yb^0,\alphab)$ nor $(\yb^1,\alphab)$ will affect $u_2(\xb) = \xb_2$, since $u_2(\zerob) = \partial_{\xb_1} u_2(\zerob) = 0$.
Then, on the second pass through the for loop, if we  again fixed  $\alphab = [1,0]$, the subproblem in Line 4 amounts to maximizing the zero polynomial, which will clearly result in failure of \Cref{alg:model_completion}.
This example illustrates how differentiation can nullify nontrivial polynomials; this motivates our flexibility to adjust which $\alphab \in A$ is paired with each pivot $u_i$ in each pass through the for loop of \Cref{alg:model_completion}.

\begin{algorithm2e}[h!]
     \SetAlgoNlRelativeSize{-3}
\caption{Birkhoff model improvement with pivot polynomials}
\label{alg:model_improvement}

\KwIn{%
$\yb^0$,
$D = \{ (\yb^i, \alphab^i) \mid i=0, \dots, p \}$
with $(\yb^0, \zerob) \in D$,
$\Delta$, $\xi_{acc} > 0$, improvement threshold $\xi_{imp} > 1$.
}
\KwOut{Data $D'$ numerically poised up to threshold $\xi_{imp} \xi_{acc}$ or certificate that such improvement is impossible with one condition switch.}

Run \Cref{alg:model_completion} on $D$ to produce $D' = \{ (\zb^i, \betab^i) \mid i=0, \dots, q \}$ and $u_q$.

$ \displaystyle
(\zb^\star, \alphab^{\star}) \gets \argmax_{(\yb,\alphab) \in \Ball(\yb^0, \Delta) \times A}
\abs{ \partial^{\alphab} u_q(\hat{\yb}) }$

\uIf{$\abs{\partial^{\alphab^\star} u_q(\hat{\zb}^\star)} > \xi_{imp} \abs{\partial^{\alphab^q} u_q(\hat{\zb}^q)}$}{
    $D' \gets D' \setminus \{ (\zb^q, \alphab^q) \} \cup \{ (\zb^\star, \alphab^\star) \}$
}
\Else{
\Return \textbf{failure}
}

\end{algorithm2e}

We have illustrated that one can always generate poised Birkhoff interpolation data corresponding to a well-defined quadratic polynomial interpolant.
However, if the conditions $D$ provided as input to \Cref{alg:model_completion} exhibit poor geometry in the sense that the poisedness constant $\Lambda$ associated with $D$ is large, and provided the acceptance threshold $\xi_{acc}$ is sufficiently small, then it is possible that \Cref{alg:model_completion} returns $D' \subseteq D$, which is also poorly poised.
To mitigate such situations, we present \Cref{alg:model_improvement}, which improves the geometry of given data $D$ by changing exactly one of the Birkhoff interpolation conditions.
\Cref{alg:model_improvement} is intended to follow \Cref{alg:model_completion}.
\Cref{alg:model_improvement} searches over the product space of the trust region and the available set for a condition that improves the last pivot polynomial.
The candidate condition is accepted only if it increases the magnitude of the final pivot value by at least a factor of the improvement threshold $\xi_{imp} > 1$.
Together, \Cref{alg:model_completion} and \Cref{alg:model_improvement} provide a framework for iteratively updating Birkhoff interpolation models.

\subsection{Algorithm Statement}

\Cref{alg:optimizer} combines \Cref{alg:model_completion} and \Cref{alg:model_improvement} into an optimization method that uses available derivative information; it is a variant of the method of Conn, Scheinberg, and Vicente~\cite[Algorithm 10.3]{CSV2009}, and \Cref{sec:converge} extends their convergence results to our setting.

The algorithm alternates between building a local Birkhoff interpolation model and computing a trust-region step from it, updating the interpolation set and derivative conditions each iteration to maintain model quality.

\begin{algorithm2e}[h!]
    \SetAlgoNlRelativeSize{-4}
\caption{Trust-region method with some available derivatives}
\label{alg:optimizer}
\KwIn{An initial point $\xb^{(0)}$ and trust-region radius $\Delta^{(0)} > 0$. Initial data $\bar{D}$.
A maximum radius $\Deltamax > \Delta^{(0)}$.
Specify the constants $\eta$, $\gammadec$, $\gammainc$, $\epsilon_c$, $\beta$, and $\mu$
which satisfy $0 \le \eta < 1$, $0 < \gammadec < 1 < \gammainc$, $\epsilon_c > 0$, and $\mu > \beta > 0$.
}

\For{$k = 0, 1, \dots$}{

    Run \Cref{alg:model_completion} on $\bar{D}$ to produce $D$.

    Evaluate $\partial^{\alphab} f(\yb)$ for each new $(\yb, \alphab)$ in $\bar{D}$.

    Build $m^{(k)}(\xb^{(k)} + \sv) = c^{(k)} + (\gb^{(k)})^T \sv + \frac{1}{2} \sv^T \Hb^{(k)} \sv$ using $D$.

    Set $\sigma^{(k)} = \max \{ \norm{\gb^{(k)}}, -\lambda_{min}(\Hb^{(k)}) \}$.\\

    \tcc{Criticality step}
    \eIf{$\sigma^{(k)} > \epsilon_c$ or $\Delta^{(k)} \le \mu \sigma^{(k)}$}{
        Proceed to step calculation.
    }{ %
    Improve $\bar{D}$ with \Cref{alg:model_improvement} to generate $\tilde{D}$ and corresponding new model $\tilde{m}^{(k)}$, trust-region radius $\tilde{\Delta}^{(k)}$, and optimality measure $\tilde{\sigma}^{(k)}$.\\
    Set
    $m^{(k)} \gets \tilde{m}^{(k)}$,
    $\sigma^{(k)} \gets \tilde{\sigma}^{(k)}$, and
    $\Delta^{(k)} \gets \min \{ \max \{ \tilde{\Delta}^{(k)}, \beta \tilde{\sigma}^{(k)} \}, \Delta^{(k)} \}$.
    }

    \tcc{Step calculation}
    Compute a step $\sv^{(k)}$ that minimizes $m^{(k)}$ over $\Ball(\xb^{(k)}, \Delta^{(k)})$.\\

    Compute $f(\xb^{(k)} + \sv^{(k)})$ and
    $\displaystyle
    \rho^{(k)} \gets \frac{f(\xb^{(k)}) - f(\xb^{(k)} + \sv^{(k)})}{m^{(k)}(\xb^{(k)}) - m^{(k)}(\xb^{(k)} + \sv^{(k)})}
    .$\\

    $\bar{D} \gets \bar{D} \cup \{ ( \xb^{(k)} + \sv^{(k)}, \zerob ) \}$. \\

    \tcc{Model update}

    \uIf{$\rho^{(k)} \ge \eta$}{
        Accept the trial point $\xb^{(k+1)} \gets \xb^{(k)} + \sv^{(k)}$.}
    \Else{
        Improve $\bar{D}$ with \Cref{alg:model_improvement} and set $\xb^{(k+1)} \gets \xb^{(k)}$.
        }

    \uIf{$\rho^{(k)} \ge \eta$ and $\norm{\sv^{(k)}} = \Delta^{(k)}$}{
        Increase the trust-region radius: $\Delta^{(k+1)} \gets \min( \gammainc \Delta^{(k)}, \Deltamax )$.
    }
    \uElseIf{$\rho^{(k)} < \eta$ and $\norm{\xb - \xb^{(k)}} \le \Delta^{(k)}$ for all $(\xb, \alphab) \in D$}{
         Decrease the trust-region radius by $\Delta^{(k+1)} \gets \gammadec \Delta^{(k)}$.
         }
     \Else{$\Delta^{(k+1)} \gets \Delta^{(k)}$}
    }
\end{algorithm2e}

\Cref{alg:optimizer} has several parameters to govern step acceptance and the trust-region radius.
The ratio $\rho^{(k)}$ compares the actual reduction $f(\xb^{(k)}) - f(\xb^{(k)} + \sv^{(k)})$ to the reduction $m^{(k)}(\xb^{(k)}) - m^{(k)}(\xb^{(k)} + \sv^{(k)})$ predicted by the model, and a step is accepted once $\rho^{(k)} \ge \eta$.
The radius $\Delta^{(k)}$ is decreased by a factor of $\gammadec \in (0,1)$ if a step is rejected and $D$ already lies within $\Ball(\xb^{(k)}, \Delta^{(k)})$.
The radius $\Delta^{(k)}$ is increased by a factor of $\gammainc > 1$ (up to at most $\Deltamax$) if an accepted step reaches the trust-region boundary.
Constants $\epsilon_c$, $\beta$, and $\mu$, with $\mu > \beta > 0$, govern the criticality step. The optimality measure $\sigma^{(k)} = \max\{ \norm{\gb^{(k)}}, -\lambda_{\min}(\Hb^{(k)}) \}$ is zero at second-order stationary points of $m^{(k)}$ and must be sufficiently large relative to $\Delta^{(k)}$ for the algorithm to proceed. 
\Cref{alg:model_improvement} is called to improve the poisedness of the interpolation data and resets $\Delta^{(k)}$ to a value between $\beta \tilde{\sigma}^{(k)}$ and its previous value.
This mirrors the criticality step of~\cite[Algorithm 10.2]{CSV2009} and is needed for the convergence results in \Cref{sec:converge}.

We determine interpolation data $D$ in each iteration with \Crefrange{alg:model_completion}{alg:model_improvement}.
\Cref{alg:model_improvement} only improves the geometry of the interpolation set; the trust-region radius reduction in the criticality step occurs only in \Cref{alg:optimizer} (as in~\cite[Algorithm 10.2]{CSV2009}).
The initial model for \Cref{alg:optimizer} can be obtained by performing \Cref{alg:model_completion}
with $D= \{ (\yb^0,\zerob) \}$
and solving the linear system \eqref{eq:m_interp_linsys_normalized} to obtain model coefficients.

\section{Convergence Analysis} \label{sec:converge}
Our analysis follows the standard trust-region paradigm used in model-based derivative-free optimization. The main contribution is demonstrating that the proposed Birkhoff interpolation models satisfy the  accuracy conditions required by existing global convergence theory.
We focus on a setting employing quadratic Birkhoff interpolation models here, but there are straightforward analogues for linear Birkhoff interpolation models.

In this section we argue that \Cref{alg:optimizer} converges to a second-order critical point of the objective function.
For this, we defer to existing analysis that a trust-region algorithm using sufficiently accurate local models yields this convergence property~\cite{CSV2009}.
Therefore, we show here that Birkhoff interpolation models provide sufficiently accurate local approximations in the sense that they are fully quadratic, defined as follows.

\begin{definition} \label{def:models}
    A model $m$ is a fully quadratic approximation of a function $f$ on $\Ball(\yb^0, \Delta)$ if there exist constants $\kappaeh$, $\kappaeg$, $\kappaef$, and $L_{\Hess m}$ for which $\Hess m$ is Lipschitz continuous with a Lipschitz constant bounded by $L_{\Hess m}$, and such that the errors between the Hessians, gradients, and values of $m$ and $f$ satisfy
    \begin{align*}
        \norm{\Hess m(\yb) - \Hess f(\yb)} \le \kappaeh \Delta,\\
        \norm{\grad m(\yb) - \grad f(\yb)} \le \kappaeg \Delta^2,\\
        | m(\yb) - f(\yb) | \le \kappaef \Delta^3,
    \end{align*}
    respectively, for all $\yb \in \Ball(\yb^0, \Delta)$.
\end{definition}

First, the choice to take a quadratic polynomial $m$ for the Birkhoff interpolant automatically satisfies the Lipschitz Hessian condition with $L_{\Hess m} = 0$.
To facilitate the convergence discussion for the remaining bounds, we introduce notation for tracking the number of interpolation conditions in $D=\{(\yb^i, \alphab^i)\}$ corresponding to each total derivative order $|\alphab^i|$.
Let $n_{\ell}$ be the number of interpolation conditions that enforce a derivative of order $\ell$. That is,
\[
n_{\ell} \defined |\{ i \mid 0 \le i \le q \text{ and } |\alphab^i| = \ell \}|.
\]
With poised data for a quadratic interpolant, we have $n_0 + n_1 + n_2 = q + 1$.
We will also write the interpolant $m$ in the form introduced in \eqref{eq:m_expression},
\begin{equation*} %
    m(\yb) = c + \gb^T (\yb - \yb^0) + \frac{1}{2} (\yb - \yb^0)^T \Hb (\yb - \yb^0).
\end{equation*}

We first state a standard proposition, without proof, derived from Taylor's theorem and Lipschitz continuity.

\begin{proposition}\label{prop:taylor}
    Let \Cref{assn:poised_D2lipshitz} hold.
    There exist functions $c_f(\yb,\zb): \cL(\xb^{(0)})\times\cL^{'}(\xb^{(0)})\to\Reals$, $\cb_g(\yb, \zb):  \cL(\xb^{(0)})\times\cL^{'}(\xb^{(0)})\to\Reals^n$ and $\Cb_H(\yb,\zb): \cL(\xb^{(0)})\times\cL^{'}(\xb^{(0)})\to\Reals^{n\times n}$
    such that
    \begin{enumerate}
        \item $f(\zb) - f(\yb) - \grad f(\yb)^\top (\zb-\yb) - \frac{1}{2} (\zb-\yb)^\top \Hess f(\yb) (\zb-\yb) = c_f(\yb,\zb)\|\zb-\yb\|^3$,
        \item $\grad f(\zb) - \grad f(\yb) - \Hess f(\yb)^\top (\zb-\yb) = \cb_g(\yb,\zb)\|\zb-\yb\|^2$, and
        \item $\Hess f(\zb) - \Hess f(\yb) = \Cb_H(\yb,\zb)\|\zb-\yb\|$
    \end{enumerate}
    such that
    $|c_f(\yb,\zb)| \leq \frac{L_{\Hess f}}{6}$,
    $\|\cb_g(\yb, \zb)\|_{\infty} \leq \frac{L_{\Hess f}}{2}$, and
    $\|\Cb_H(\yb,\zb)\|_{\infty} \leq L_{\Hess f}$,
    each bound holding over the respective function's domain.
\end{proposition}

We are now ready to show that a quadratic polynomial Birkhoff interpolant is a fully quadratic  approximation to any underlying function that is sufficiently regular and bounded, in the sense of \Cref{assn:poised_D2lipshitz}.
The following result is an analog for the accuracy of Lagrange interpolants~\cite[Theorem 3.16]{CSV2009}.

In practice, the interpolation sets used by \Cref{alg:optimizer} are generated and maintained by repeated calls to \Crefrange{alg:model_completion}{alg:model_improvement}.
These two methods ensure $D$ is poised within the current trust region and contains sufficiently many interpolation conditions to determine a quadratic model.
Under the regularity assumptions of \Crefrange{assn:A_D}{assn:poised_D2lipshitz}, the models constructed throughout \Cref{alg:optimizer} satisfy the hypotheses required by the following theorem whenever \Cref{alg:model_improvement} terminates successfully.

\begin{theorem} \label{thm:Birkhoff_fully_quadratic}
    Let \Crefrange{assn:A_D}{assn:poised_D2lipshitz} hold.
    Let a poised set of interpolation data, $D$, a trust radius $\Delta > 0$, and a point
    $\yb^0 \in \cL(\xb^{(0)})$ %
    be given. Assume $Y$ satisfies $\Delta(Y) \le \Delta \le \Deltamax$.
    Then the Birkhoff interpolation model $m$ described by \eqref{eq:m_expression}, and \eqref{eq:m_interp} is fully quadratic on $\Ball(\yb^0, \Delta)$.
    That is, $m$ satisfies \Cref{def:models} with
    \begin{align*}
        \kappaeh &= \sqrt{2} \left( \frac{9}{4} n_{0} + 4 n_{1} + 4 n_{2} \right)^{1/2} L_{\Hess f} \norm{\hat{\Mb}^{-1}},
        \\
        \kappaeg &= (1 + \sqrt{2}) \left( \frac{9}{4} n_{0} + 4 n_{1} + 4 n_{2} \right)^{1/2} L_{\Hess f} \norm{\hat{\Mb}^{-1}},
        \text{ and}
        \\
        \kappaef &= \left( \frac{4}{3} + 2 \left(1 + 2 \sqrt{2} \right) \left( \frac{9}{4} n_{0} + 4 n_{1} + 4 n_{2} \right)^{1/2} \norm{\hat{\Mb}^{-1}} \right) L_{\Hess f}.
    \end{align*}
\end{theorem}

\begin{proof}
Let $\yb \in \Ball(\yb^0, \Delta)$ be given.
We explicitly write the modeling errors in the function value, gradient, and Hessian as
\begin{align}
   \label{eq:func_id} m(\yb) &= f(\yb) + e^f(\yb) \\
    \label{eq:grad_id} \grad m(\yb) &= \gb + \Hb (\yb - \yb^0) = \grad f(\yb) + \tilde{\eb}^g(\yb) \\
    \label{eq:hess_id} \Hess m(\yb) &= \Hb = \Hess f(\yb) + \tilde{\Eb}^H(\yb).
\end{align}
We separate the analysis based on the total orders of the multi-indices $\{\alphab^i \mid i = 0, \dots, q\}$ in $D$.
Without loss of generality, permute $i=0,1,\dots,q$ so that the orders $|\alphab^i|$ are nondecreasing, that is, all indices with $|\alphab^i|=0$ appear first in the reordering, followed by those with $|\alphab^i|=1$ and then followed by those with $|\alphab^i|=2$.

\paragraph{\textbf{Case 1:} $\alphab^i=\zerob$}
We have from \eqref{eq:m_interp} that $m(\yb^i) = f(\yb^i)$ and so, by \eqref{eq:func_id}, $e^f(\yb^i)=0$.
Hence
$m(\yb^i) - m(\yb) = f(\yb^i) - (f(\yb) + e^f(\yb))$, and therefore
\begin{align}
    f(\yb^i) - &f(\yb) - e^f(\yb) \nonumber\\
    &= \gb^T (\yb^i - \yb) + \frac{1}{2} (\yb^i - \yb^0)^T \Hb (\yb^i - \yb^0) - \frac{1}{2} (\yb - \yb^0)^T \Hb (\yb - \yb^0) \nonumber\\
    &= \gb^T (\yb^i - \yb) + \frac{1}{2} (\yb^i - \yb)^T \Hb (\yb^i - \yb) + (\yb^i - \yb)^T \Hb (\yb - \yb^0), \label{eq:zero_property}
\end{align}
where the last line comes from some algebraic manipulation.
Using \eqref{eq:grad_id} and \eqref{eq:hess_id}, we have that
\begin{align}
    f(\yb^i) - &f(\yb) - \grad f(\yb)^T (\yb^i - \yb) - \frac{1}{2} (\yb^i - \yb)^T \Hess f(\yb) (\yb^i - \yb) - e^f(\yb) \nonumber \\
    &= \tilde{\eb}^g(\yb)^T (\yb^i - \yb) + \frac{1}{2} (\yb^i - \yb)^T \tilde{\Eb}^H(\yb) (\yb^i - \yb) \nonumber \\
    \label{eq:error_identity} &\implies c_f(\yb^i, \yb)\|\yb^i-\yb\|^3 - e^f(\yb) = (\yb^i - \yb)^T (\tilde{\eb}^g(\yb) + \frac{1}{2} \tilde{\Eb}^H(\yb) (\yb^i - \yb) ),
\end{align}
where $c_f$ is from \Cref{prop:taylor}.
Subtracting \eqref{eq:error_identity} with $\yb^0$ from \eqref{eq:error_identity} with $\yb^i$, we have that
\begin{align*}
    &c_f(\yb^i, \yb)\|\yb^i-\yb\|^3 - c_f(\yb^0, \yb)\|\yb^0-\yb\|^3 \\
    &=  (\yb^i - \yb^0)^T \tilde{\eb}^g(\yb) + \frac{1}{2} (\yb^i - \yb)^T \tilde{\Eb}^H(\yb) (\yb^i - \yb) - \frac{1}{2} (\yb^0 - \yb)^T \tilde{\Eb}^H(\yb) (\yb^0 - \yb).
\end{align*}
Denoting
\begin{equation} \label{eq:def_t}
    \tb(\yb) \defined \tilde{\eb}^g(\yb) - \tilde{\Eb}^H(\yb) (\yb - \yb^0)
\end{equation}
and rearranging, we conclude that, for all $\yb \in \Reals^n$,
\begin{equation} \label{eq:id_deg_0}
    (\yb^i - \yb^0)^T \tb(\yb) + \frac{1}{2} (\yb^i - \yb^0)^T \tilde{\Eb}^H(\yb) (\yb^i - \yb^0)
    = c_f(\yb^i, \yb)\|\yb^i-\yb\|^3-
    c_f(\yb^0, \yb)\|\yb^0-\yb\|^3.
\end{equation}

\paragraph{\textbf{Case 2:} $|\alphab^i|=1$}
Let $k_i$ be the unique index such that $\alphab^i_{k_i} = 1$.
By \eqref{eq:m_interp},
\[
\partial_{x_{k_i}} m(\yb^i)
= (\eb^{k_i})^T (\gb + \Hb (\yb^i - \yb^0))
= \partial_{x_{k_i}} f(\yb^i).
\]
From \eqref{eq:grad_id}, $\partial_{x_{k_i}} m(\yb) = (\eb^{k_i})^T (\grad f(\yb) + \tilde{\eb}^g(\yb))$, which we subtract from both sides to yield
\[
(\eb^{k_i})^T\Hb(\yb^i-\yb)
= (\eb^{k_i})^T [\grad f(\yb^i)-\grad f(\yb)-\tilde{\eb}^g(\yb)].
\]
Using \eqref{eq:hess_id} and \Cref{prop:taylor}, we obtain
\begin{align*}
    & (\eb^{k_i})^T[\nabla^2 f(\yb) + \tilde{\Eb}^H(\yb)](\yb^i-\yb)
    = (\eb^{k_i})^T [\grad f(\yb^i)-\grad f(\yb)-\tilde{\eb}^g(\yb)]\\
    \implies & (\eb^{k_i})^T\tilde{\Eb}^H(\yb)(\yb^i-\yb)
    = (\eb^{k_i})^T [\grad f(\yb^i)-\grad f(\yb)-\nabla^2 f(\yb)(\yb^i-\yb)-\tilde{\eb}^g(\yb)]\\
    \implies & (\eb^{k_i})^T\big[\tilde{\eb}^g(\yb)+\tilde{\Eb}^H(\yb)(\yb^i-\yb)\big]=(\eb^{k_i})^T\cb_g(\yb,\yb^i)\|\yb^i-\yb\|^2.
\end{align*}
Writing $\yb^i-\yb=(\yb^i-\yb^0)-(\yb-\yb^0)$ and using $\tb(\yb)$ as in \eqref{eq:def_t}, we get

\begin{equation} \label{eq:id_deg_1}
    (\eb^{k_i})^T \tb(\yb)+(\eb^{k_i})^T \tilde{\Eb}^H(\yb) (\yb^i - \yb^0) = (\eb^{k_i})^T\cb_g(\yb,\yb^i)\|\yb^i-\yb\|^2 
    .
\end{equation}

\paragraph{\textbf{Case 3:} $|\alphab^i|=2$}
Let $k_i$ and $k'_i$ be two indices for which the condition
\[
\partial_{x_{k_i}} \partial_{x_{k'_i}} m(\yb^i)
= (\eb^{k_i})^T \Hb \eb^{k'_i}
= \partial_{x_{k_i}} \partial_{x_{k'_i}} f(\yb^i)
\]
holds.
Using \eqref{eq:hess_id}, we obtain
\[
(\eb^{k_i})^T (\Hess f(\yb^i)-\Hess f(\yb) - \tilde{\Eb}^H(\yb)) \eb^{k'_i}=0.
\]
By \Cref{prop:taylor},
\begin{equation} \label{eq:id_deg_2}
    (\eb^{k_i})^T \tilde{\Eb}^H(\yb) \eb^{k'_i} = (\eb^{k_i})^T \Cb_H(\yb^i,\yb)\eb^{k'_i}\|\yb^i-\yb\|.
\end{equation}

With these three cases complete, we can take
equations \eqref{eq:id_deg_0}, \eqref{eq:id_deg_1}, and \eqref{eq:id_deg_2} to form a linear system to determine the unknown vector $\tb(\yb)$ defined in \eqref{eq:def_t} and symmetric matrix $\tilde{\Eb}^H(\yb)$.
To express them together in matrix form, we introduce the mapping $(\xb, \zb) \mapsto \wb(\xb, \zb)$, which we define as the vector satisfying
\[
\wb(\xb, \zb)^T \textvec{\Bb} \defined \xb^T \Bb \zb
\]
for any arbitrary symmetric matrix $\Bb$ of appropriate size, keeping in mind that $\textvec{\Bb}$ only enumerates the upper triangular entries of $\Bb$. Explicitly, we have $\wb(\xb, \zb) = \textvec{\Wb(\xb, \zb)}$, where $\Wb(\xb, \zb)$ is the matrix with entries
\[
W_{i j}(\xb, \zb) = \begin{cases}
    x_i z_i, \quad &i = j \\
    x_i z_j + x_j z_i, \quad &i < j.
\end{cases}
\]
Thus, the system \eqref{eq:id_deg_0}, \eqref{eq:id_deg_1}, and \eqref{eq:id_deg_2} may be expressed as
\begin{equation} \label{eq:model_error_linsys}
    \begin{bNiceArray}{c|c}
        (\yb^1 - \yb^0)^T & \frac{1}{2} \wb(\yb^{1} - \yb^0, \yb^{1} - \yb^0)^T \\
        \vdots & \vdots \\
        (\yb^{n_{0} - 1} - \yb^0)^T & \frac{1}{2} \wb(\yb^{n_{0} - 1} - \yb^0, \yb^{n_{0} - 1} - \yb^0)^T \\
        \hline
        (\eb^{k_{n_{0}}})^T & \wb(\eb^{k_{n_{0}}}, \yb^{n_{0}} - \yb^0)^T \\
        \vdots & \vdots \\
        (\eb^{k_{n_{0} + n_{1} - 1}})^T & \wb(\eb^{k_{n_{0} + n_{1} - 1}}, \yb^{n_{0} + n_{1} - 1} - \yb^0)^T \\
        \hline
        \zerob^T & \wb(\eb^{k_{n_{0} + n_{1}}}, \eb^{k'_{n_{0} + n_{1}}})^T \\
        \vdots & \vdots \\
        \zerob^T & \wb(\eb^{k_{q}}, \eb^{k'_{q}})^T
    \end{bNiceArray}
    \begin{bNiceArray}{c}
        \tb(\yb) \\
        \textvec{\tilde{\Eb}^H(\yb)}
    \end{bNiceArray}
    =
    \begin{bNiceArray}{c}
        \ab(\yb) \\
        \bb(\yb) \\
        \cb(\yb)
    \end{bNiceArray},
\end{equation}
where $\ab(\yb), \bb(\yb), \cb(\yb)$ contain the right-hand side entries of \eqref{eq:id_deg_0}, \eqref{eq:id_deg_1}, and \eqref{eq:id_deg_2}.

We intend to derive bounds on $\tb(\yb)$ and $\textvec{\tilde{\Eb}^H(\yb)}$ in \eqref{eq:model_error_linsys}.
This is most easily done by first rescaling the system in \eqref{eq:model_error_linsys}.
Note that each entry of  the components of the right-hand side  involves norm terms of the form $\|\yb-\yb^i\|$ or $\|\yb-\yb^0\|$, which can be upper-bounded by $2\Delta(Y)$ and $\Delta(Y)$, respectively.
Thus, using \Cref{prop:taylor}, each entry of
$\ab(\yb)$ is bounded in absolute value by $\frac{3}{2}L_{\Hess f}\Delta(Y)^3$,
each entry of $\bb(\yb)$ is bounded in absolute value by $4L_{\Hess f}\Delta(Y)^2$,
and each entry of $\cb(\yb)$ is bounded in absolute value by $2L_{\Hess f}\Delta(Y)$.
Let $\Qb$ denote the system equations \eqref{eq:model_error_linsys}.
Consider the rescaled matrix
\[
\hat{\Qb} \defined
\begin{bmatrix}
    \Ib_{n_{0} - 1} & & \\
    & \Delta(Y) \Ib_{n_{1}} & \\
    & & \Delta(Y)^2 \Ib_{n_{2}}
\end{bmatrix}
\Qb
\begin{bmatrix}
    \Delta(Y)^{-1} \Ib_{n} & \\
    & \Delta(Y)^{-2} \Ib_{n(n+1)/2}
\end{bmatrix}.
\]
Note that \eqref{eq:model_error_linsys} immediately yields the bound
\begin{equation*}
    \norm{
    \begin{bmatrix}
        \Delta(Y) \tb(\yb) \\
        \Delta(Y)^2 \textvec{\tilde{\Eb}^H(\yb)}
    \end{bmatrix}
    }
    \le
    \norm{\hat{\Qb}^{-1}}
    \norm{
    \begin{bmatrix}
        \ab(\yb) \\
        \Delta(Y) \bb(\yb) \\
        \Delta(Y)^2 \cb(\yb)
    \end{bmatrix}
    }.
\end{equation*}
From our bounds on the individual norms of $\ab(\yb)$, $\bb(\yb)$, and
$\cb(\yb)$, we have that
\begin{align*}
    \norm{
    \begin{bmatrix}
        \ab(\yb) \\
        \Delta(Y) \bb(\yb) \\
        \Delta(Y)^2 \cb(\yb)
    \end{bmatrix}
    }^2
    &= \norm{\ab(\yb)}^2 + \Delta(Y)^2 \norm{\bb(\yb)}^2 + \Delta(Y)^4 \norm{\cb(\yb)}^2 \\
    &\le \frac{9}{4} n_{0} L_{\Hess f}^2 \Delta(Y)^6
        + 4 n_{1} L_{\Hess f}^2 \Delta(Y)^6
        + 4 n_{2} L_{\Hess f}^2 \Delta(Y)^6
        \\
    &= \left( \frac{9}{4} n_{0} + 4 n_{1} + 4 n_{2} \right) L_{\Hess f}^2 \Delta(Y)^6.
\end{align*}
From here, we first derive our bound on $\tilde{\Eb}^H(\yb)$.
We have that
\begin{align*}
    \norm{\Delta(Y)^2 \textvec{\tilde{\Eb}^H(\yb)}}
    &\le \norm{\begin{bmatrix} \Delta(Y) \tb(\yb) \\ \Delta(Y)^2 \textvec{\tilde{\Eb}^H(\yb)} \end{bmatrix}} \\
    &\le \left( \frac{9}{4} n_{0} + 4 n_{1} + 4 n_{2} \right)^{1/2} L_{\Hess f} \norm{\hat{\Qb}^{-1}} \Delta(Y)^3 \\
    \implies \norm{\textvec{\tilde{\Eb}^H(\yb)}} &\le \left( \frac{9}{4} n_{0} + 4 n_{1} + 4 n_{2} \right)^{1/2} L_{\Hess f} \norm{\hat{\Qb}^{-1}} \Delta(Y),
\end{align*}
from which we have
\begin{align*}
\norm{\tilde{\Eb}^H(\yb)}_2
\le \norm{\tilde{\Eb}^H(\yb)}_F &\le \sqrt{2} \norm{\textvec{\tilde{\Eb}^H(\yb)}}_2 \\
&\le \sqrt{2} \left( \frac{9}{4} n_{0} + 4 n_{1} + 4 n_{2} \right)^{1/2} L_{\Hess f} \norm{\hat{\Qb}^{-1}} \Delta(Y).
\end{align*}
This establishes the desired bound on $\|\tilde{\Eb}^H(\yb)\|$.
As for a bound on $\|\vec{t}(\yb)\|$, we proceed identically and find
\begin{align*}
    \norm{\Delta(Y) \tb(\yb)}
    &\le \left( \frac{9}{4} n_{0} + 4 n_{1} + 4 n_{2} \right)^{1/2} L_{\Hess f} \norm{\hat{\Qb}^{-1}} \Delta(Y)^3.
\end{align*}
From the definition $\tb(\yb) = \tilde{\eb}^g(\yb) - \tilde{\Eb}^H(\yb) (\yb - \yb^0)$, we obtain
\begin{align*}
    \norm{\tilde{\eb}^g(\yb)} &\le \norm{\tb(\yb)} + \norm{\tilde{\Eb}^H(\yb)} \norm{\yb - \yb^0} \\
    &\le (1 + \sqrt{2}) \left( \frac{9}{4} n_{0} + 4 n_{1} + 4 n_{2} \right)^{1/2} L_{\Hess f} \norm{\hat{\Qb}^{-1}} \Delta(Y)^2.
\end{align*}
Finally, for $i$ with $|\alphab^i| = 0$, we can combine \eqref{eq:zero_property} with
the expression $\gb = \grad f(\yb) + \tilde{\eb}^g(\yb) - \Hb (\yb - \yb^0)$ to conclude
\begin{align*}
    e^f(\yb)
    =\;& f(\yb^i) - f(\yb) - (\yb^i - \yb)^T (\grad f(\yb) + \tilde{\eb}^g(\yb) - \Hb (\yb - \yb^0)) \\
        &- \frac{1}{2} (\yb^i - \yb)^T \Hb (\yb^i - \yb) - (\yb^i - \yb)^T \Hb (\yb - \yb^0)
    \\
    =\;& f(\yb^i) - f(\yb) - (\yb^i - \yb)^T \grad f(\yb) - \frac{1}{2} (\yb^i - \yb)^T \Hess f(\yb) (\yb^i - \yb) \\
        &- (\yb^i - \yb)^T \tilde{\eb}^g(\yb) - \frac{1}{2} (\yb^i - \yb)^T \tilde{\Eb}^H(\yb) (\yb^i - \yb)
    .
\end{align*}
Again using the bound on the absolute value of entries of $\ab(\yb)$,
\begin{align*}
    |e^f(\yb)|
    &\le \frac{4}{3} L_{\Hess f} \Delta(Y)^3 + (2 \Delta(Y)) \norm{\tilde{\eb}^g(\yb)} + \frac{1}{2} (2 \Delta(Y))^2 \norm{\tilde{\Eb}^H(\yb)}
    \\
    &\le \left( \frac{4}{3} + \left( 2(1+\sqrt{2}) + 2\sqrt{2} \right) \left( \frac{9}{4} n_{0} + 4 n_{1} + 4 n_{2} \right)^{1/2} \norm{\hat{\Qb}^{-1}} \right) L_{\Hess f} \Delta(Y)^3 \\
    &\le \left( \frac{4}{3} + 2 \left( 1 + 2 \sqrt{2} \right) \left( \frac{9}{4} n_{0} + 4 n_{1} + 4 n_{2} \right)^{1/2} \norm{\hat{\Qb}^{-1}} \right) L_{\Hess f} \Delta(Y)^3
    .
\end{align*}

Now, the matrix $\hat{\Qb}$
is related to the matrix $\hat{\Mb}$ from \eqref{eq:Mhat} via
\[
\hat{\Mb} =
\begin{bmatrix}
    1 & \zerob^T \\
    \vec{z} & \hat{\Qb}
\end{bmatrix},
\]
where $\vec{z}=[1, \dots, 1, 0, \dots, 0]^T$ with $(n_{0} - 1)$ many leading ones.

Since $D$ is poised, $\hat{\Mb}$ is invertible. We also have that $\hat{\Qb}$ is invertible;  if $\hat{\Qb}$ were singular, then $\hat{\Mb}$ would be singular as well because of the zero row above $\Qb$.
From here, block matrix inversion yields
\[
\hat{\Mb}^{-1} =
\begin{bmatrix}
    1 & \zerob^T \\
    -\hat{\Qb}^{-1} \vec{z} & \hat{\Qb}^{-1}
\end{bmatrix}.
\]
Since $\hat{\Qb}^{-1}$ is a submatrix of $\hat{\Mb}^{-1}$, we deduce that $\normil{\hat{\Qb}^{-1}} \le \normil{\hat{\Mb}^{-1}}$.

We conclude that the model is fully quadratic with constants
\begin{align*}
    \kappaeh &= \sqrt{2} \left( \frac{9}{4} n_{0} + 4 n_{1} + 4 n_{2} \right)^{1/2} L_{\Hess f} \norm{\hat{\Qb}^{-1}},
    \\
    \kappaeg &= (1 + \sqrt{2}) \left( \frac{9}{4} n_{0} + 4 n_{1} + 4 n_{2} \right)^{1/2} L_{\Hess f} \norm{\hat{\Qb}^{-1}},
    \text{ and}
    \\
    \kappaef &= \left( \frac{4}{3} + 2 \left(1 + 2 \sqrt{2} \right) \left( \frac{9}{4} n_{0} + 4 n_{1} + 4 n_{2} \right)^{1/2} \norm{\hat{\Qb}^{-1}} \right) L_{\Hess f}.
\end{align*}
Replacing $\normil{\hat{\Qb}^{-1}}$ with $\normil{\hat{\Mb}^{-1}}$ gives the desired constants.
\end{proof}

This result is a specific form of the error bounds developed by Ciarlet and Raviart~\cite{Ciarlet_1972} but with exact constants and a different approach to the proof.
From our bounds in \Cref{sec:birkhoff_interp}, we also have the following associated corollary.
\begin{corollary}
    Let \Cref{assn:poised_D2lipshitz} hold. Then the Birkhoff
    interpolation model $m$ given by \eqref{eq:m_expression} is fully quadratic
    on $\Ball(\yb^0, \Delta)$ with the constants
    \begin{align*}
        \kappaef &= \left( \frac{4}{3} + 2 \left(1 + 2 \sqrt{2} \right) \left( \frac{9}{4} n_{0} + 4 n_{1} + 4 n_{2} \right)^{1/2} C \Lambda \right) L_{\Hess f},
        \\
        \kappaeg &= (1 + \sqrt{2}) \left( \frac{9}{4} n_{0} + 4 n_{1} + 4 n_{2} \right)^{1/2} L_{\Hess f} C \Lambda,
        \text{ and}
        \\
        \kappaeh &= \sqrt{2} \left( \frac{9}{4} n_{0} + 4 n_{1} + 4 n_{2} \right)^{1/2} L_{\Hess f} C \Lambda
    \end{align*}
    for a constant $C = C(n, \phi)$, which depends only on the input dimension and the choice of the natural basis.
\end{corollary}

\begin{proof}
\Cref{thm:Lambda_iff_invnorm} establishes that $\normil{\hat{\Mb}^{-1}} \le C \Lambda$, and the result follows immediately from \Cref{thm:Birkhoff_fully_quadratic}
by replacing $\normil{\hat{\Mb}^{-1}}$ with $C \Lambda$.

\end{proof}

The model-improvement procedures
maintain interpolation sets $D$ that guarantee
the models $m$ remain fully quadratic models throughout the optimization process.
 With fully quadratic models available on every iteration, it is established that the model-based trust-region framework, of which \Cref{alg:optimizer} is a particular instance, yields a sequence of model centers converging to a second-order stationary point whenever the model Hessians norms $\normil{\Hb^{(k)}}$ are bounded above uniformly for all iterations $k$~\cite{CSV2009}.

\section{Experimental Results}
\label{sec:experiments}

We now investigate the performance of the proposed Birkhoff
interpolation framework in settings with limited derivative availability.

\subsection{CUTEst Numerical Experiments}

We test a Python implementation of \Cref{alg:optimizer} on a subset of problems from the CUTEst problem collection~\cite{Gould2015}.
While CUTEst provides (full) first- and second-derivative information for all tested problems (which we do use for benchmarking), we simulate availability of only some derivatives via the following approach.
Given an unconstrained objective function $f$ with domain $\Reals^n$, we generate a subset of $\{1,\dots,n\}$ for which the corresponding first partial derivatives of $f$ can be evaluated,
    $K \defined \{ k \mid \partial_{\xb_k}f\}$.
The remaining partial derivatives are unavailable, that is, $U \defined \{1,\dots,n\}\setminus K$.
As for second derivative availability, we assume that both $k, k' \in K$ if and only if $\partial_{\xb_k} \partial_{\xb_{k'}} f$ is available.
That is, in our notation, the available derivative set is
\begin{equation} \label{eq:def_A}
A = \left\{  \alphab \mid \alphab_j = 0 \text{ for all } j \in U, |\alphab| \le 2  \right\}.
\end{equation}
In our experiments we consider different levels of derivative availability, parameterized by a fraction of $K$.
Precisely,
for each objective function $f$ in the test set (with $n$-dimensional domain), we generate uniformly at random without replacement a subset $K$ from $\{1,\dots,n\}$; the size of $K$ is varied such that $\lceil n_K/n\rceil \in \{0.25, 0.5, 0.75\}$,
 where we denote $n_K \triangleq |K|$.
 After the random draw is made, the induced some-available-derivatives problem is then held fixed in this test set; randomness was  introduced only to avoid systematic bias in the generation of test problems.

We selected 89 CUTEst problems to assess the convergence properties and numerical performance of \Cref{alg:optimizer}.
The criteria for selecting these functions are that they are unconstrained, have $n \le 15$, and are twice continuously differentiable  and that all considered solvers completed successfully.
We note that there is no guarantee that $\Hess f$ is Lipschitz continuous for all $f$, and so \Cref{assn:poised_D2lipshitz} may not hold for such $f$.
For each problem, we run our solver with the different available derivative sets $A$ described above.
We consider a problem solved when the solver identifies a point $\xb^{\star}$ satisfying $\normil{\grad f(\xb^{\star})} \le \tau$;
for benchmarking purposes, we use the true gradient $\grad f$.

Our data profiles measure computational effort in terms of normalized oracle accesses, as opposed to iteration counts or function evaluations alone.
The horizontal axes report the number of unique point/multi-index queries $(\xb,\alphab)$ divided by $n+1$.
Thus, evaluating a function value and evaluating an available derivative component at the same point are counted as distinct units of information; repeated requests for an already computed quantity are not recounted.
Consequently, methods with larger $K$ incur higher oracle costs per sampled point.
Each evaluation unit is a single evaluation associated with a datum $(\yb,\alphab)$; this amounts to assuming that the cost of evaluating any partial derivative $\partial^{\alphab} f$ is equal to the cost of evaluating $f$; this may not be true in practice, but makes the presentation of results substantially more interpretable. 
In the experiments presented, we focus on the intermediate derivative-availability regimes $n_K/n \in \{0.25,0.50,0.75\}$, since the fully derivative-free and fully derivative-based settings are less informative for comparing the relative advantages of partial-derivative interpolation models.

We compare our limited derivative algorithm with another such algorithm described by Fuhrl\"{a}nder and Sch\"{o}ps~\cite{Fuhrlander2023}.
There, given a set of points $Y$, the authors begin with interpolation conditions $D_0 = \{ (\yb^i, \zerob) : i = 0, \dots, p \}$ for each $\yb^i\in Y$.
They then augment the set $D_0$ to $D_1 = \{ (\yb, \alphab) : (\yb, \zerob) \in D_0, \alphab \in A, \abs{\alphab} \ge 1 \}$ to obtain the Hermite interpolation conditions $D = D_0 \cup D_1$.
That is, given any $\alphab \in A$, the value $\partial^{\alphab} f(\xb)$ is queried if and only if $\partial^{\alphab'} f(\xb)$ is also queried for every $\alphab' \in A$.
The number of points $p$ ($|Y|=p$) is chosen so that $\abs{D} \ge q+1$; the system \eqref{eq:m_interp_linsys} is then solved in the least-squares regression sense.

The method of Fuhrl\"{a}nder and Sch\"{o}ps, being based on \texttt{BOBYQA}, is also a trust-region method.
In our experiments, we employed identical parameters common to trust-region methods.
We set stopping criteria in both methods to small values: methods were not stopped unless the model gradient norm was machine epsilon.

We display the results of this numerical study in \Cref{fig:CUTEst}.
We refer to our method simply as ``Birkhoff" interpolation, and we refer to the method of Fuhrl\"{a}nder and Sch\"{o}ps simply as ``Hermite" interpolation, in order to highlight the key difference in interpolation approaches.
Overall, we observe that increasing the proportion of known derivatives generally improves optimization performance for both interpolation approaches.
In particular, the Birkhoff interpolation method exhibits a relatively consistent improvement as the ratio $n_K/n$ increases from $0.25$ to $0.75$.
This suggests that the incorporation of derivative information can substantially improve model quality without requiring complete gradient access.

\begin{figure}
    \centering

    \begin{minipage}{0.49\linewidth}
        \centering
        \includegraphics[width=\linewidth]{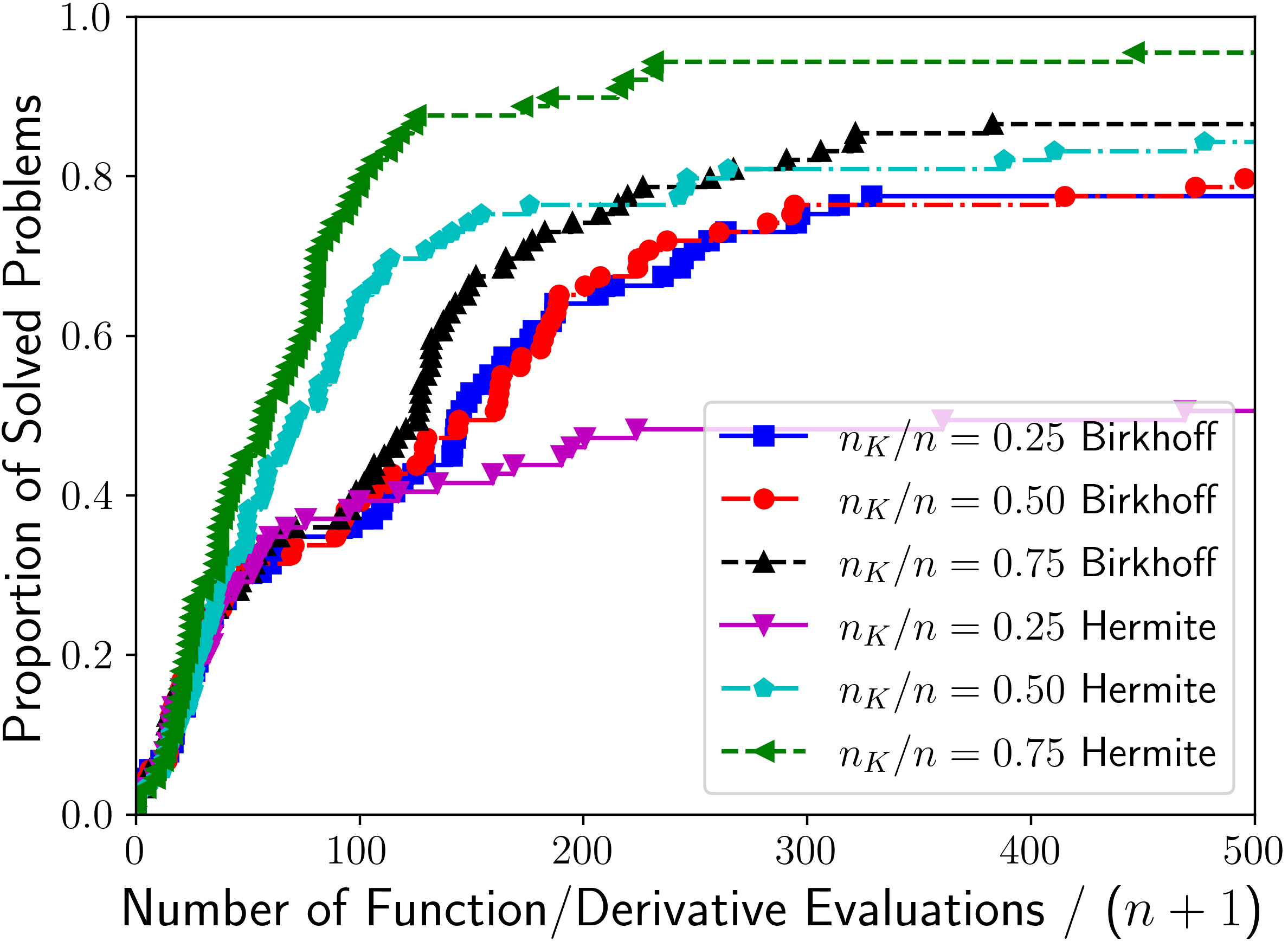}

        $\tau = 10^{-2}$
    \end{minipage}
    \hfill
    \begin{minipage}{0.49\linewidth}
        \centering
        \includegraphics[width=\linewidth]{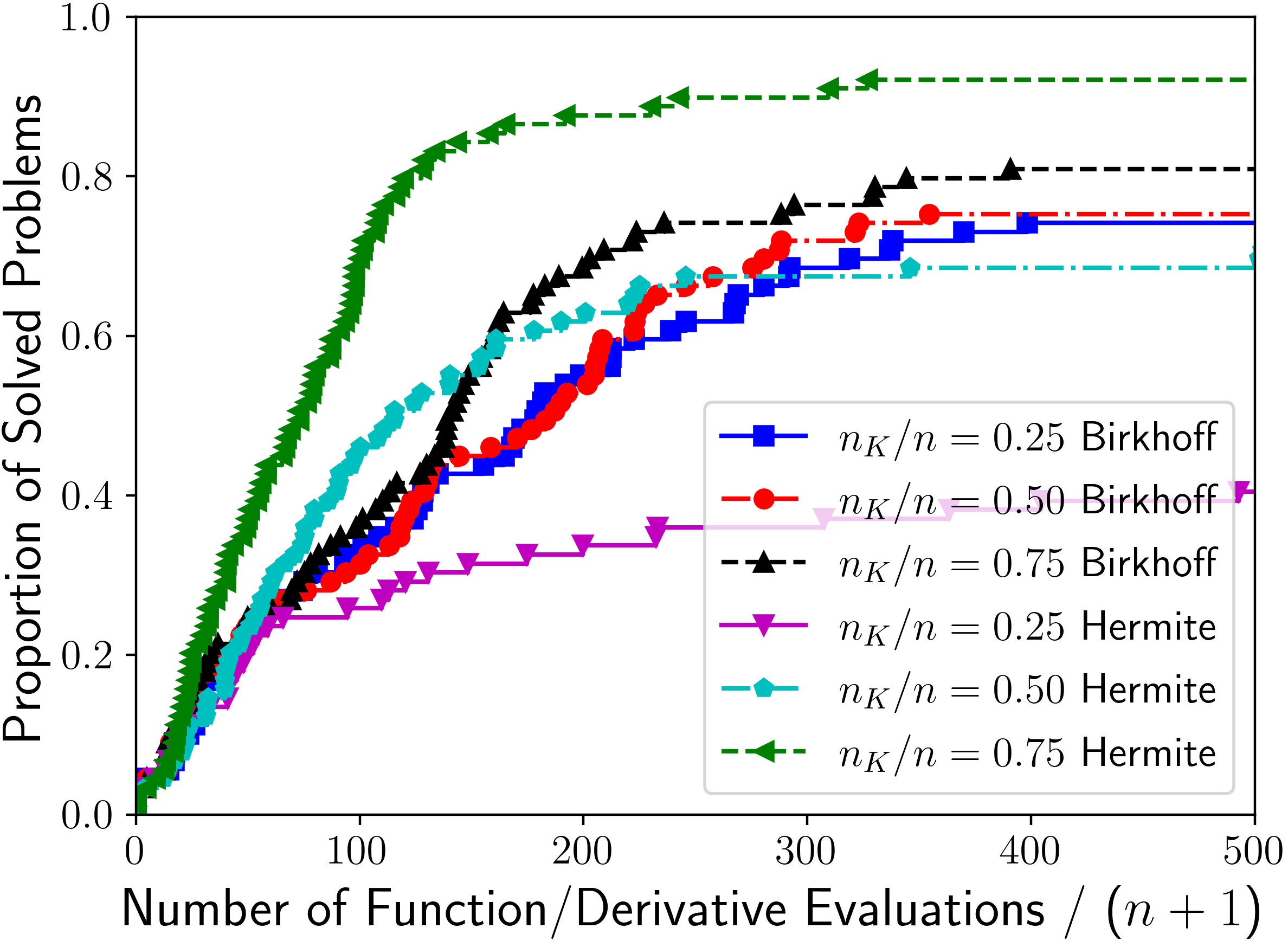}

        $\tau = 10^{-4}$
    \end{minipage}
    \caption{
    Performance profiles on 89 CUTEst problems comparing our Birkhoff approach and Hermite~\cite{Fuhrlander2023} interpolation approaches under varying levels of derivative availability. The quantity $n_K/n$ denotes the fraction of variables assigned to the known derivative set $K$.}
    \label{fig:CUTEst}
\end{figure}

The Hermite approach~\cite{Fuhrlander2023} demonstrates stronger sensitivity to the amount of derivative information available.
In particular, the Hermite method performs exceptionally well when $75\%$ of the derivatives are known, solving a large proportion of the test problems within relatively small oracle budgets.
Its performance degrades substantially, however, when fewer derivatives are available, especially in the $25\%$ case.
The Birkhoff approach exhibits less variability across the different derivative-availability regimes.

One possible reason for this difference is the manner in which interpolation conditions are incorporated into the model-building process.
The Hermite approach includes all derivative conditions associated with a point simultaneously, whereas the Birkhoff framework permits more selective derivative usage patterns.
Moreover, our implementation of the Birkhoff interpolation method consistently constructs fully quadratic interpolation models; this requires
$q_n = (n+1)(n+2)/2$
data in $D$.
Consequently, future developments of our Birkhoff method may benefit from greater flexibility, in particular by employing underdetermined quadratic models; see, for example,~\cite[Chapter 5]{CSV2009}.
For similar reasons,
we did not include purely derivative-free or fully derivative-based optimization strategies. The Hermite implementation reduces to \texttt{BOBYQA} in the absence of derivative information, whereas our implementation always constructs fully quadratic models; this is an extreme disadvantage.

We believe that these results indicate that partial derivative information can significantly improve optimization efficiency but that the manner in which derivative information is incorporated into the interpolation conditions can strongly influence robustness and performance.

\subsection{A Real-World Stellarator Design Problem} \label{sec:stellarator_example}

To complement the above CUTEst numerical tests, we also present a
some-available-derivatives problem arising from stellarator design.
We build on the single-stage optimization approach of Jorge et
al.~\cite{Jorge2023} implemented in the \texttt{simsopt} software package~\cite{Landreman_2021}. The
objective combines a term depending on the shape of the plasma boundary,
(requiring an expensive physics-based equilibrium solve) and a term
depending on the shapes of the surrounding magnetic coils, $f(\xb) =
f_1(\xb_{\text{surf}}) + w f_2(\xb_{\text{coil}}, \xb_{\text{surf}})$, where
$\xb_{\text{coil}}$ and $\xb_{\text{surf}}$ denote coil and plasma-boundary
parameters, respectively, and $w$ is a fixed weighting constant.
We construct a reduced $n=10$-dimensional
test instance by fixing all but $n_K=2$ coil parameters and all $8$
surface parameters at values known to perform well in this objective, leaving us with
$\xb\in\Reals^{10}$ as the vector of free coil and surface parameters. 
Derivatives of $f$ with respect to the retained coil parameters are
cheap to obtain via an adjoint calculation, while derivatives with respect to the $8$
surface parameters are unavailable. 
We run \Cref{alg:optimizer} and the Hermite method of
Fuhrl\"{a}nder and Sch\"{o}ps~\cite{Fuhrlander2023} using the same starting
point $\xb_0$, the same initial trust-region radius and the same
budget of $500(n+1)$ evaluations.
As before, one evaluation corresponds to a single point visited by the
trust-region method. The Hermite method queries all $n_K$ known derivative
components together and the function value at every point it visits. 
The Birkhoff approach query only a subset of
the available derivative components at a given point. Therefore, the two curves in
\Cref{fig:stellarator} accumulate distinct (point, derivative condition) pairs at 
different rates even though both methods are run for the same number of evaluations.

\begin{figure}[h!]
    \centering
    \includegraphics[width=0.7\linewidth]{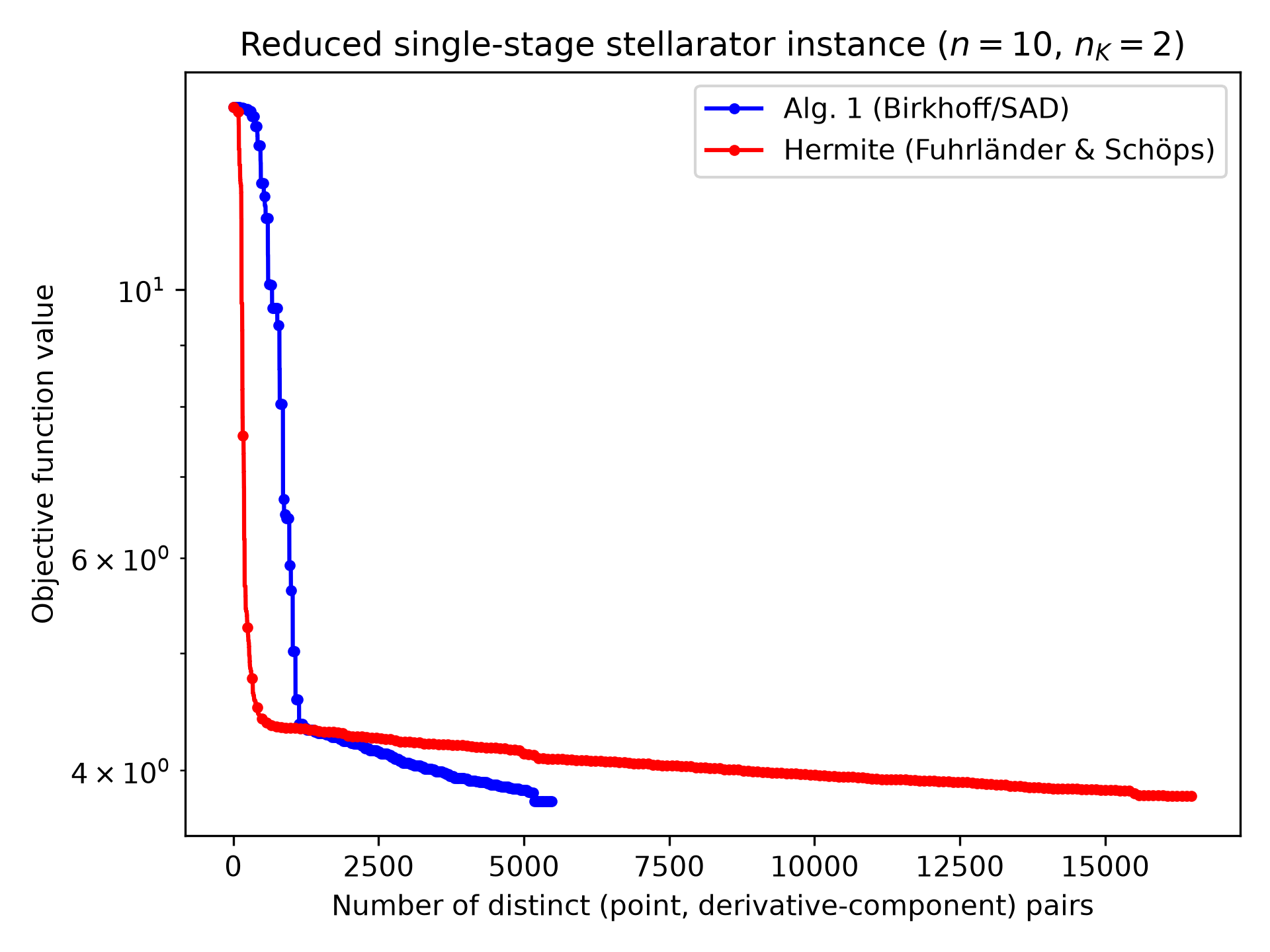}
    \caption{Results on the stellarator design problem described in \Cref{sec:stellarator_example}, comparing our Birkhoff approach to the Hermite method of \cite{Fuhrlander2023}. The horizontal axis reports the number of distinct (point, derivative condition) pairs queried.}
    \label{fig:stellarator}
\end{figure}

As shown in \Cref{fig:stellarator}, the Birkhoff approach reaches a slightly
better final objective value ($3.772$) than the Hermite approach ($3.810$), and
does so while querying substantially fewer distinct (point,
derivative-component). This matches with the qualitative behavior
observed on the CUTEst collection in \Cref{fig:CUTEst}: on this genuine
simulation-based objective, exploiting partial derivative information through
Birkhoff interpolation is at least competitive with, and here modestly
outperforms, the Hermite approach.

\section{Conclusions}
\label{sec:conclusions}

We have introduced a derivative-free optimization framework based on
Birkhoff interpolation models and established guarantees of model quality in the
unconstrained setting. Our approach expands the flexibility of
interpolation-based models by allowing derivative information to be
incorporated in a dynamic manner, while preserving the theoretical guarantees
needed for global convergence. Numerical experiments demonstrate that the
resulting models can be constructed effectively and perform competitively
across a range of test problems.

More broadly, our framework provides a bridge between classical derivative-free optimization and modern simulation environments in which derivative information is only partially available.
We expect such settings to become increasingly common in large-scale scientific computing applications involving coupled multiphysics models, embedded machine learning components, and legacy simulation infrastructure.

Although we focused on the unconstrained case, bound constraints can be added to the trust-region subproblem without difficulty. It is challenging to adapt interpolation-model machinery to more general constrained settings. 
Even the simpler case of Lagrange interpolation (where $A=\{\zerob\}$) is nontrivial and is the subject of recent work on convex-constrained model construction~\cite{Hough2022,Roberts2025}.

\paragraph{Code availability.} An implementation of
\Cref{alg:optimizer}, together with the scripts needed to reproduce the
Birkhoff runs on the CUTEst tests in \Cref{fig:CUTEst} and the
stellarator example of \Cref{sec:stellarator_example}, is available at
\url{https://github.com/POptUS/SAD}. The Hermite least-squares baseline of
\cite{Fuhrlander2023} used in both comparisons was run using code provided by
that paper's authors.

\bibliographystyle{siamplain}
\bibliography{references}
\end{document}